\documentclass[11pt,a4paper]{amsart}
\usepackage{amsmath,amsfonts,xargs,amssymb}
\usepackage{float}
\usepackage[utf8]{inputenc}
\usepackage{graphicx} % Required for inserting images
\usepackage{amsthm}
\usepackage{tikz}
\usetikzlibrary{calc}
\usepackage{tkz-tab}
\usepackage{theoremref}
\usepackage{thmtools}
\declaretheorem{theorem}
\usepackage{hyperref}
\usepackage{comment}
\usepackage{mathtools}
\usepackage{subcaption}
\mathtoolsset{showonlyrefs, showmanualtags}
\hypersetup{colorlinks=true,urlcolor=black,
pdftitle=''Volume and Projection Inequalities I: Zonoids and Courtade's Conjecture''}

\usepackage[T1]{fontenc}
\usepackage{lmodern}
\usepackage{microtype}
\usepackage{booktabs, array,longtable}
\usepackage{enumitem}

\theoremstyle{plain}

\newtheorem{lemma}[theorem]{Lemma}

\newtheorem{prop}[theorem]{Proposition}

\newtheorem{conj}{Conjecture}

\theoremstyle{definition}
\newtheorem{definition}[theorem]{Definition}

\newtheorem{remark}{Remark}

\newcommand{\R}{\mathbb R}
\newcommand{\s}{\mathbb S}

\DeclareMathOperator{\V}{V}

\title[Volume and Projection Inequalities I]{Volume and Projection Inequalities I:\\ Zonoids and Courtade's Conjecture}
\author[M. Fradelizi, A. Hubard, A. Manui, C. S. Ndiaye, S. Wang and A. Zvavitch]{Matthieu Fradelizi, Alfredo Hubard, Auttawich Manui,\\
Cheikh Saliou Ndiaye, Shouda Wang and Artem Zvavitch}

\date{}


\begin{document}

\begin{abstract}

We study volume and projection inequalities for zonoids through the multiaffine determinant polynomials that encode their volumes.  We show that  a log-submodularity conjecture for the volume of zonoids is equivalent to the Rayleigh property of zonotope volume polynomials, which we prove for the case when the degree or codegree is at most  3.  

This result is sharp in that when  the degree and codegree are at least 4, we construct counterexamples using the existence of non-Rayleigh matroids in ranks at least four.
Additionally, we provide unimodular or graphical counterexamples in dimensions four and higher to an equivalent projection inequality formulation of the conjecture. 
We also show that a stronger projection inequality fails even in dimension three.

We next disprove Courtade's
conjecture using a pair of orthogonal double bodies of revolution.  Although Courtade's conjecture was originally formulated for general convex bodies, we show that it fails even for zonoids in every dimension at least three.

\end{abstract}

\maketitle
\tableofcontents

\section{Introduction}

The goal of this work is to study volume and projection inequalities for zonoids inspired by the Pl\"unnecke--Ruzsa inequality~\cite{H-70,I-89,TV06:book} from additive combinatorics. The inequality  states that for finite subsets $A,B,C$ of a commutative group, there exists a nonempty set
$X\subset A$ such that
\[
 (\#A)^2\#(X+B+C)\leq \#X\,\#(A+B)\,\#(A+C)
\]
holds.
Ruzsa~\cite{Ruz97} extended this principle to compact subsets of locally compact
commutative groups equipped with Haar measure.

In~\cite{BM-12}, Bobkov and Madiman first studied the Pl\"unnecke--Ruzsa inequality in the context of convex geometry, where they established a volume inequality for convex bodies with a non-sharp constant factor depending on the dimension.  The factor was later improved in~\cite{FMZ-24}; see also \cite{FMMZ-24,MNZ-25, MMZZ-26} for related studies.

While it was known that the optimal constant in the Pl\"unnecke--Ruzsa inequality for convex bodies grows at least exponentially with dimension (see  \cite[Thm.~4.11]{FMZ-24} and \cite{NT-17}), it was believed that the optimal constant would be $1$ for a special class of convex bodies called zonoids \cite[Conj.~4.13]{FMZ-24}. In the present paper, we disprove this conjecture in dimensions at least four. Furthermore, we disprove two other related conjectures  proposed
in~\cite{DCT-91, FMMZ-24}.

Recall that a zonotope is a finite Minkowski sum of segments, and a zonoid is a Hausdorff limit of zonotopes. Throughout the paper, a convex body means a compact convex set. When $K \subset \mathbb{R}^n$, we write $|K|$ for its $n$-dimensional Lebesgue measure; in particular, $|K|=0$ if $K$ is lower dimensional. If $K$ is contained in an $m$-dimensional affine subspace, then $|K|_m$ denotes its $m$-dimensional Hausdorff measure in its affine subspace. For a linear subspace $E, P_E$ denotes the orthogonal projection onto $E$.
\subsection{Log-submodularity of volumes}
The following conjecture was formulated
in~\cite{FMMZ-24}.
Note that item (i) below is analogous to the  Pl\"unnecke--Ruzsa inequality and can also be interpreted as the log-submodularity of the volume of zonoids. 
\begin{conj}
\label{conj:Plun-Ruz_zonoids}
The following equivalent statements  hold:
\begin{enumerate}[label=\textup{(\roman*)}]
\item For all zonoids $A,B,C \subset \R^n$,
\begin{equation}\label{eq:Plun-Ruz_zonoids}
 |A|\,|A+B+C|\leq |A+B|\,|A+C|. 
\end{equation}
\item For every full-dimensional zonoid $A$, every zonoid $B$,
and every $u\in\s^{n-1}$,
\begin{equation}\label{eq:conjweak}
 \frac{|A|}{|P_{u^\perp}A|_{n-1}}
 \leq
 \frac{|A+B|}{|P_{u^\perp}(A+B)|_{n-1}}.
\end{equation}
\item For every zonoid $A$ and every pair of orthonormal vectors
$u,v\in\s^{n-1}$,
\begin{equation}\label{eq:equivlent-conj}
|A|\,|P_{\{u,v\}^\perp}A|_{n-2}
 \leq 
 |P_{u^\perp}A|_{n-1}|P_{v^\perp}A|_{n-1}.
\end{equation}
\end{enumerate}
\end{conj}

The equivalence of the three formulations was proved in~\cite{FMMZ-24}, which also established Conjecture~\ref{conj:Plun-Ruz_zonoids} in
dimensions two and three.
It is also shown in~\cite{FMMZ-24} that (iii) in Conjecture~\ref{conj:Plun-Ruz_zonoids}
holds in $\R^n$ when $A$ is generated by $n$ segments. See also~\cite{AS-25, ADRS-24} for a recent alternative algebraic-combinatorial proof. For arbitrary convex bodies, Fenchel's inequality \cite{Fen36} yields \eqref{eq:equivlent-conj} with the constant $2(n-1)/n$, which is sharp in the class of all convex bodies. 

However, we show that Conjecture~\ref{conj:Plun-Ruz_zonoids} fails in every dimension $n\geq 4$. The examples presented in the present paper are zonotopes with the smallest possible number of generators.

 Our approach is based on the observation that  Conjecture~\ref{conj:Plun-Ruz_zonoids} is equivalent to the Rayleigh property of zonotope volume polynomials, which we explain below. First, we recall that for a zonotope $Z = \sum_{i=1}^m[0,u_i]$, its volume can be expressed by the following formula,
\begin{equation}\label{eq:zonotope-volume-formula}
 |Z|= \sum_{|I|=n}
 \left|\det(u_i)_{i\in I}\right|.
\end{equation}
Fix $v_1,\ldots,v_m\in\R^n$. For $x \in {\mathbb R}^m_+$, consider the zonotope  $Z_x=\sum_{i=1}^m x_i[0,v_i]$, and the polynomial $f(x)=|Z_x|$ given by \eqref{eq:zonotope-volume-formula}:
\begin{equation}\label{eq:zonotope-volume}
f(x)=\sum_{|I|=n}
 |\det(v_i)_{i\in I}|\prod_{i\in I}x_i.
\end{equation}
Such a polynomial $f$ is called a \emph{zonotope volume polynomial}.
By construction, $f$ is a homogeneous multiaffine polynomial with nonnegative coefficients.

\begin{definition}\label{def:rayleigh}
A homogeneous multiaffine polynomial with nonnegative coefficients is called \emph{Rayleigh} if
\begin{equation}
 \partial_i f(x)\partial_j f(x)
 \geq f(x)\partial_{ij}f(x), \quad \mbox{  where  } x\in\R_+^m,\ i\neq j.  \label{eq:rayleigh-intro}
\end{equation}
\end{definition}

We will  show in Section \ref{sec:rayleigh-positive}  the following strong equivalence of Conjecture \ref{conj:Plun-Ruz_zonoids} with the Rayleigh property: fix $m\geq n$, the conditions 
\begin{itemize}
\renewcommand\labelitemi{\tiny$\bullet$}
    \item for every collection of vectors $v_1,\dots, v_m\in \R^n$, $f(x) = |\sum_{i=1}^m x_i[0,v_i]|$ is Rayleigh,
    \item (i) in Conjecture \ref{conj:Plun-Ruz_zonoids} holds for zonotopes $A,B,C$ generated by a common set of $m$ vectors,
    \item (ii) in Conjecture \ref{conj:Plun-Ruz_zonoids} holds for full-dimensional zonotopes $A,B$ generated by a common set of $m-1$ vectors,
    \item (iii) in Conjecture \ref{conj:Plun-Ruz_zonoids} holds for full-dimensional zonotopes $A$ generated by a common set of $m-2$ vectors,
\end{itemize}
are equivalent.

The Rayleigh polynomial viewpoint naturally suggests looking at Rayleigh matroids. Indeed, replacing every nonzero coefficient $|\det(v_i)_{i\in I}|$ in~\eqref{eq:zonotope-volume} by $1$ gives the basis-generating polynomial of the matroid represented by $v_1,\ldots,v_m$; see~\cite{Oxley}.
A matroid is called Rayleigh if its basis-generating polynomial is Rayleigh~\cite{CW}.
It is known that all matroids with rank at most three are Rayleigh~\cite {wagner}. However, there exist rank-four $\R$-representable non-Rayleigh matroids~\cite{CW}.
This suggested to us to look at the representations of such matroids for a counterexample to Conjecture \ref{conj:Plun-Ruz_zonoids}.

It turned out that, indeed, the representation of the non-Rayleigh matroid from~\cite{CW}  gives a counterexample to the Rayleigh property of zonotope volume polynomials, disproving Conjecture \ref{conj:Plun-Ruz_zonoids}.
And in fact, the number of generators in this zonotope example is sharp, that is, every zonotope volume polynomial with fewer generators than this example is indeed Rayleigh.
More precisely, we have the following theorem.

\begin{theorem}\label{thm:rayleigh-threshold}
Let $m\geq n$. If $\min\{n,m-n\}\leq3$, then 
\begin{itemize}
    \item[1).] for every collection of vectors $v_1,\dots, v_m\in \R^n$, the zonotope volume polynomial 
    $$f(x) = |\sum_{i=1}^m x_i[0,v_i]|$$ 
    is a Rayleigh polynomial;
    \item[2).] for every collection of vectors $v_1,\dots, v_m\in \R^n$, the log-submodularity inequality  
    $$
    |A||A+B+C|\leq |A+B||A+C|
    $$ 
    holds for all zonotopes $A,B,C$ generated by $v_1,\dots,v_m$;
    \item[3).]  for every collection of vectors $v_1,\dots, v_{m-1}\in \R^n$, 
    $$\frac{|A|}{|P_{u^\perp}A|_{n-1}}
 \leq
     \frac{|A+B|}{|P_{u^\perp}(A+B)|_{n-1}}$$ holds for zonotopes $A,B$  (where $A$ is full-dimensional)  generated by $v_1,\dots, v_{m-1}$ and every vector $u\in S^{n-1}$;
     \item[4).] for every collection of vectors $v_1,\dots, v_{m-2}\in \R^n$,  
     $$
     |A|\,|P_{\{u,v\}^\perp}A|_{n-2}\leq |P_{u^\perp}A|_{n-1}|P_{v^\perp}A|_{n-1}
     $$
     holds for all zonotopes $A$ generated by $v_1,\dots, v_{m-2}$ and all orthogonal vectors $u,v\in S^{n-1}$.
\end{itemize}
If $\min\{n,m-n\}>3$, then every item above is false.

\end{theorem}

\begin{remark} 
Claim~\textit{4).}   in Theorem~\ref{thm:rayleigh-threshold} was first obtained by Averkov, von Dichter, and Soprunov~\cite{ADS}   using a different method.
\end{remark}

Note that $n$ is the degree of $f$ in \textit{1).} and $m$ is the number of variables of $f$.  We will  call  $m-n$ the \emph{codegree} of $f$.

  The assertions for the $n\leq3$ case follow directly from  Conjecture~\ref{conj:Plun-Ruz_zonoids} in dimensions at most three and, when $m-n\leq3$, they follow from Gale duality of maximal minors.  The proof will be given in Section \ref{sec:rayleigh-positive}.

  In addition to the counterexamples showing the negative claim in Theorem \ref{thm:rayleigh-threshold}, we provide examples violating~\eqref{eq:equivlent-conj} which are unimodular, or even graphical (and hence totally unimodular). These counterexamples will be given in Section \ref{sec:zonoid-volume}.

\begin{theorem}\label{thm:equivlent-conj}
\begin{itemize}
\item[]
    \item[] 1). There is a counterexample to  inequality~\eqref{eq:equivlent-conj} in $\R^4$ given by a zonotope with six generators.
    
    \item[] 2). There is a counterexample to  inequality~\eqref{eq:equivlent-conj} in $\R^4$ given by a unimodular matrix with seven columns. 
    
    \item[] 3). There is a counterexample to  inequality~\eqref{eq:equivlent-conj} in $\R^5$ given by a graphical and totally unimodular matrix with seven columns. 
\end{itemize}

\end{theorem}

\subsection{A Stronger Conjecture on Projection Volumes}
We next consider the stronger, two-term projection inequality proposed
in~\cite{DCT-91, FMMZ-24}.

\begin{conj}\label{conj:str_conj}
 For full-dimensional zonoids $A,B\subset\R^n$ and
$u\in\s^{n-1}$,
\begin{equation}\label{eq:strong-context}
 \frac{|A+B|}{|P_{u^\perp}(A+B)|_{n-1}}
 \geq
 \frac{|A|}{|P_{u^\perp}A|_{n-1}}
 +\frac{|B|}{|P_{u^\perp}B|_{n-1}}.
\end{equation}
\end{conj}

This is the projection analogue of the following conjecture of
Dembo, Cover, and Thomas~\cite{DCT-91}; within the class of
full-dimensional zonoids, the two statements are equivalent
by~\cite[Remark~3.10]{FMMZ-24}:
\begin{equation}\label{conj:DCT}
 \frac{|A+B|}{|\partial(A+B)|_{n-1}}
 \geq
 \frac{|A|}{|\partial A|_{n-1}}
 +\frac{|B|}{|\partial B|_{n-1}}.
\end{equation}
Bonnesen's inequality implies Conjecture~\ref{conj:str_conj} in the
plane~\cite{B-1929,Schneider-14}.  The next theorem completes
the dimensional picture and disproves Conjecture~\ref{conj:str_conj}, and hence
\eqref{conj:DCT}, in every dimension $n\geq3$.

\begin{theorem}\label{thm:str_conj}
Conjecture~\ref{conj:str_conj} is false in every dimension
$n\geq3$.
\end{theorem}

\subsection{Courtade's conjecture}
Although Conjecture~\ref{conj:Plun-Ruz_zonoids} holds in dimension three, the
following stronger inequality proposed by Courtade does not.  

\begin{conj}\label{conj:courtade}
Let $n\geq2$, and let $B_2^n$ be the Euclidean unit ball.  For all
convex bodies $B,C\subset\R^n$,
\begin{equation}\label{eq:courtade-intro}
 (|B|\,|C|)^{1/n}+(|B_2^n|\,|B_2^n+B+C|)^{1/n}
 \leq |B_2^n+B|^{1/n}|B_2^n+C|^{1/n}.
\end{equation}
\end{conj}
Courtade's conjecture was proved in dimension two
in~\cite{FMMZ-24}.  We show that it fails in every dimension $n\geq3$.
The common mechanism is the necessary surface-area condition proved in
Proposition~\ref{prop:courtade-surface}: applying \eqref{eq:courtade-intro} to
$tB,tC$ and letting $t\to\infty$ gives
\begin{equation}\label{eq:surface-obstruction-intro}
 \left(\frac{|\partial B|_{n-1}}{|B|}+\frac{|\partial C|_{n-1}}{|C|}\right)^n
 \geq n^n|B_2^n|\frac{|B+C|}{|B|\,|C|}.
\end{equation}
Two orthogonal double bodies of revolution in $\R^3$ reverse
\eqref{eq:surface-obstruction-intro}; all of the required volumes and
surface areas are computed geometrically. Thus, this first example already
disproves Conjecture~\ref{conj:courtade} in dimension three.

We then strengthen the conclusion within the class of zonoids.
In dimension three, we construct an unconditional six-generator zonotope
$B$ and take $C$ to be its quarter-turn.  Congruence reduces
\eqref{eq:surface-obstruction-intro} to a single inequality involving
$|B|$, $|B+C|$, and $|\partial B|_2$.  Grouped determinant and cross-product
computations give exact values for all three quantities and strictly reverse
that inequality.  Proposition~\ref{prop:courtade-surface} then shows that
$tB,tC$ violate Courtade's original inequality for every sufficiently large
$t$.

For \(n\geq4\), we begin with a smooth zero-mass perturbation
of the standard generating measure of \(B_2^4\). For sufficiently
small values of the perturbation parameter, the resulting
measures remain positive and therefore generate zonoids
whose sum is \(2B_2^4\). A fourth-order expansion strictly reverses
\eqref{eq:surface-obstruction-intro}. The determinant formula for
the volume of a zonoid transfers the two relevant coefficients
to every higher dimension, and spherical Poisson smoothing
produces smooth  zonoids while preserving the strict
reverse inequality. Proposition~\ref{prop:courtade-surface} then
yields counterexamples to Courtade's inequality after sufficiently
large common dilations.

The paper is organized as follows.  Section~\ref{sec:rayleigh-positive}
establishes the positive Rayleigh range.  Section~\ref{sec:zonoid-volume}
contains the Rayleigh counterexamples and their few-generator, unimodular,
and graphical refinements, followed by the strong-projection counterexample.
Section~\ref{sec:courtade} derives the necessary surface-area condition
and gives counterexamples first among centrally symmetric convex bodies and
then among zonoids.

\medskip
\noindent\emph{Note added.}
 As this manuscript was being completed, we learned of the independent work of
Skorupinski~\cite{S-26}, who obtained, by a different method, a
four-dimensional counterexample to \textup{(i)} in
Conjecture~\ref{conj:Plun-Ruz_zonoids}, also with the smallest possible number of generators.  Skorupinski also proves the conjecture when the full combined
matrix
generating $A+B+C$ is unimodular and characterizes the equality cases in that
setting.  The results and counterexamples in the present paper were obtained
independently.  In addition to different explicit counterexamples, the
present work establishes the rank--corank threshold above and treats the
strong projection inequality and Courtade's conjecture.

\section{Zonotope volume polynomials of degree or codegree at most three}
\label{sec:rayleigh-positive}
We start with a quantitative version of the equivalence between the Rayleigh property and the assertions in Conjecture \ref{conj:Plun-Ruz_zonoids}.
\begin{prop}
\label{prop: equiv}
Fix \(m\geq n\). The following conditions are equivalent:
\begin{itemize}
\renewcommand\labelitemi{\tiny$\bullet$}
    \item for every \(v_1,\dots,v_m\in\R^n\), the volume polynomial $
        f(x)=\left|\sum_{i=1}^m x_i[0,v_i]\right|$ 
    is Rayleigh;
    \item part~\textup{(i)} of Conjecture~\ref{conj:Plun-Ruz_zonoids}
    holds for zonotopes \(A,B,C\) generated by a common set of \(m\)
    vectors;
    \item part~\textup{(ii)} of
    Conjecture~\ref{conj:Plun-Ruz_zonoids} holds for 
    zonotopes $A$ and  $B$ generated by a common set of \(m-1\) vectors, where $A$ is assumed to be full dimensional;
    \item part~\textup{(iii)} of
    Conjecture~\ref{conj:Plun-Ruz_zonoids} holds for full-dimensional
    zonotopes \(A\) generated by \(m-2\) vectors.
\end{itemize}
\end{prop}

\begin{proof}
Let \((\mathrm R)_m\) denote the first condition, and let
\(\textup{(i)}_m\), \(\textup{(ii)}_{m-1}\), and
\(\textup{(iii)}_{m-2}\) denote the remaining three conditions. We prove
\[
(\mathrm R)_m
\Longleftrightarrow \textup{(i)}_m
\Longrightarrow \textup{(ii)}_{m-1}
\Longrightarrow \textup{(iii)}_{m-2}
\Longrightarrow (\mathrm R)_m.
\]

\textit{Step 1}: \((\mathrm R)_m \Rightarrow \textup{(i)}_m\).
Assume  $(\mathrm R)_m$ is true. Let \(A,B,C\) be zonotopes with generators $v_1,\dots, v_m$. Then we may
write (up to translations),
\[
A=Z_x,\qquad B=Z_y,\qquad C=Z_z,
\]
where $Z_x=\sum_{i=1}^m x_i[0,v_i]$, $x,y,z\in\R_+^m$.
Let \(f(x)=|Z_x|\) be the associated zonotope volume polynomial. Since \(f\) is multiaffine, using  $(\mathrm R)_m$ we obtain that for all $x,y,z\in \R_+^m$,
\begin{align}\label{eq:directional-rayleigh-intro}
D_yf(x)D_zf(x)-f(x)D_yD_zf(x)
&=\sum_{i=1}^m y_i z_i\bigl(\partial_i f(x)\bigr)^2 \\
&\quad+\sum_{i\neq j}y_i z_j
\bigl(
\partial_i f(x)\partial_j f(x)
-f(x)\partial_{ij}f(x)
\bigr)\geq 0
\end{align}
holds.
Fix  a choice of $x,y,z\in \R_+^m$ and let $g(s,t)=f(x+sy+tz)$. Assume first $f(x)>0$.
Then \eqref{eq:directional-rayleigh-intro} implies 
\begin{equation}
\label{eq:second-derivative-log}
\partial_s\partial_t\log g(s,t)
=
\frac{g(s,t)\partial_s\partial_tg(s,t)
-\partial_sg(s,t)\partial_tg(s,t)}
{g(s,t)^2}
\leq 0.
\end{equation}
Integrating this inequality over \([0,1]^2\) gives $\log g(1,1)-\log g(1,0)-\log g(0,1)+\log g(0,0)\leq 0$, which is
\[
f(x)f(x+y+z)\leq f(x+y)f(x+z).
\]
Now if $f(x)=0$, then the above equation holds automatically.
We have thus proven  $\textup{(i)}_m$.

\textit{Step 2:}  \((\mathrm R)_m \Leftarrow \textup{(i)}_m\).
Assume $\textup{(i)}_m$ holds. 
Let $g(s,t)=f(x+sy+tz)$ as in \textit{Step 1}.
Again, if $f(x)=0$, then $
\partial_i f(x)\partial_jf(x)-f(x)\partial_{ij}f(x) =\partial_i f(x)\partial_jf(x) \geq 0$ holds automatically. Assume from now on that $f(x)>0$.
Then $\textup{(i)}_m$  implies $\log g(s,t)-\log g(s,0)-\log g(0,t)+\log g(0,0)\leq 0$ holds for all $s,t\geq 0$.
Using 
$$
\lim_{s,t\to 0}
\frac{\log g(s,t)-\log g(s,0)-\log g(0,t)+\log g(0,0)}{st} = \partial_{s}|_{s=0}\partial_t|_{t=0} \log g(s,t),
$$
and by \eqref{eq:second-derivative-log},
we obtain 
$$
D_yf(x)D_zf(x)-f(x)D_yD_zf(x)\geq 0,
$$
taking $y$ and $z$ to be coordinate vectors $e_i$ and $e_j$ shows  $
\partial_i f(x)\partial_jf(x)-f(x)\partial_{ij}f(x)\geq 0$,
thus we have shown $(\mathrm R)_m$.

\textit{Step 3:} $\textup{(i)}_m\Rightarrow\textup{(ii)}_{m-1}$.
Assume \(\textup{(i)}_m\). Let \(A,B\) be generated by a common set of
\(m-1\) vectors, and let \(u\) be a unit vector. Applying
\(\textup{(i)}_m\) to \(A,B\), and the additional segment \([0,u]\), gives
\[
|A|\,|A+B+[0,u]|
\leq
|A+B|\,|A+[0,u]|.
\]
Recall that $|K+[0,u]|=|K|+|P_{u^\perp}K|_{n-1}$. Expanding the above formula
and canceling the common term \(|A|\,|A+B|\), we obtain
\[
|A|\,|P_{u^\perp}(A+B)|_{n-1}
\leq
|A+B|\,|P_{u^\perp}A|_{n-1},
\]
which is $\textup{(ii)}_{m-1}.$
 
\textit{Step 4:} $ \textup{(ii)}_{m-1}
\Rightarrow \textup{(iii)}_{m-2}$.
Assume \(\textup{(ii)}_{m-1}\). Let \(A\) be a full-dimensional
zonotope generated by \(m-2\) vectors, and let \(u,v\) be orthonormal vectors.
Then \(A\) and \([0,v]\) are  generated by a common set of \(m-1\) vectors. Applying $\textup{(ii)}_{m-1}$  gives
\begin{equation}
    \label{eq:step4}
|A|\,|P_{u^\perp}(A+[0,v])|_{n-1}
\leq
|A+[0,v]|\,|P_{u^\perp}A|_{n-1}.
\end{equation}
Since \(u\perp v\),
\begin{align*}
|A+[0,v]|
&=|A|+|P_{v^\perp}A|_{n-1},\\
|P_{u^\perp}(A+[0,v])|_{n-1}
&=|P_{u^\perp}A|_{n-1}
  +|P_{\{u,v\}^\perp}A|_{n-2}.
\end{align*}
Substituting these into \eqref{eq:step4}, we obtain
\[
|A|\,|P_{\{u,v\}^\perp}A|_{n-2}
\leq
|P_{u^\perp}A|_{n-1}
|P_{v^\perp}A|_{n-1},
\] 
which is $\textup{(iii)}_{m-2}$.

\emph{Step 5: $(\mathrm{iii})_{m-2}\Rightarrow(\mathrm{R})_m$.}
Assume $(\mathrm{iii})_{m-2}$ and fix $i\neq j$. Since $f$ is
multiaffine,
\[
\Delta_{ij}f=\partial_i f\,\partial_j f-f\,\partial_{ij}f
\]
is independent of $x_i,x_j$. We may therefore set $x_i=x_j=0$ and
write
\[
A=\sum_{k\neq i,j}x_k[0,v_k].
\]
If $v_i,v_j$ are linearly dependent, then $\partial_{ij}f=0$, so
$\Delta_{ij}f\geq0$. Otherwise, after an invertible linear change of
coordinates, which multiplies $\Delta_{ij}f$ by a positive factor, we
may assume that $v_i,v_j$ are orthonormal. Then
\[
f(x)=|A|,\qquad
\partial_i f(x)=|P_{v_i^\perp}A|_{n-1},\qquad
\partial_j f(x)=|P_{v_j^\perp}A|_{n-1},\qquad
\partial_{ij}f(x)=|P_{\{v_i,v_j\}^\perp}A|_{n-2}.
\]
If $A$ is full-dimensional, $(\mathrm{iii})_{m-2}$ gives
$\Delta_{ij}f\geq 0$. If $A$ is lower-dimensional, then
$f(x)=|A|=0$, and hence
$\Delta_{ij}f(x)=\partial_i f(x)\partial_j f(x)\geq 0$ automatically holds.
Thus $f$ is Rayleigh.
\end{proof}

We now prove the case where $\min (n,m-n)\leq 3$ in Theorem~\ref{thm:rayleigh-threshold}.  Conjecture~\ref{conj:Plun-Ruz_zonoids}
was proved in dimensions at most three in~\cite{FMMZ-24}, hence the case
$n\leq3$ follows immediately.  The argument for the case $m-n\leq 3$ has three ingredients.
First, the Rayleigh property is preserved by duality.  
Second,
when $m>n$, Gale duality realizes the dual of a rank-$n$ zonotope volume polynomial in $m$ variables as a zonotope volume polynomial in dimension $m-n$.
Finally, we apply the three-dimensional result to the dual volume polynomial, obtaining the Rayleigh property in the case $m-n\leq 3$.

\begin{lemma}\label{lem:rayleigh-duality}
Let
$
f(y)=\sum_{J\subset[m]}c_Jy^J
$
be a homogeneous multiaffine polynomial with nonnegative coefficients, and define its
dual polynomial by
\[
f^\vee(x)=\sum_{I\subset[m]}c_{I^c}x^I.
\]
Then $f$ is Rayleigh if and only if $f^\vee$ is Rayleigh.
\end{lemma}

\begin{proof}
For distinct $i,j$, write
\[
\Delta_{ij}f
=
\partial_i f\,\partial_j f-f\,\partial_{ij}f.
\]
For $x\in\R_{>0}^m$, set $
x^{-1}=(x_1^{-1},\ldots,x_m^{-1})$.
By multiaffinity,
\[
f^\vee(x)=\prod_{k=1}^m x_k f(x^{-1}).
\]
Differentiating this identity gives
\[
\partial_i f^\vee(x)
=
\frac{\prod_{k=1}^m x_k}{x_i}
\bigl(f(x^{-1})-x_i^{-1}\partial_i f(x^{-1})\bigr)
\]
and, for $i\neq j$,
\[
\partial_{ij}f^\vee(x)
=
\frac{\prod_{k=1}^m x_k}{x_ix_j}
\bigl(
f(x^{-1})-x_i^{-1}\partial_i f(x^{-1})-x_j^{-1}\partial_j f(x^{-1})
+x_i^{-1}x_j^{-1}\partial_{ij}f(x^{-1})
\bigr).
\]
Consequently,
\[
\Delta_{ij}f^\vee(x)
=
\frac{\left(\prod_{k=1}^m x_k\right)^2}{x_i^2x_j^2}\,
\Delta_{ij}f(x^{-1}).
\]
Thus, if $f$ is Rayleigh, then $f^\vee$ is Rayleigh on
$\R_{>0}^m$, and hence on $\R_+^m$ by continuity. The converse follows
by applying the same argument to $(f^\vee)^\vee=f$.
\end{proof}

In the following lemma, we show that the dual of a zonotope volume polynomial is still a zonotope volume polynomial. A crucial linear algebra fact used in the proof comes from Gale duality.
For background on Gale duality and
Gale transforms, we refer to~\cite[Chapter~6]{Z-95}.  

\begin{lemma}
\label{lem:Gale}
Let $f$ be the volume polynomial of $m$ line segments in $\R^n$, where
$m\geq n$.  If $m=n$, then $f^\vee$ is a nonnegative constant.
If $m>n$, then $f^\vee$ is the volume polynomial of $m$ line segments in
$\R^{m-n}$.
\end{lemma}

\begin{proof}
Write the generators as the columns of an $n\times m$ matrix
$\mathcal V=(v_1\ \cdots\ v_m)$, so that
\begin{equation}\label{eq:gale-volume-polynomial}
 f(x)=\sum_{|I|=n}
 |\det \mathcal V_I|\,x^I.
\end{equation}
If $m=n$, then \eqref{eq:gale-volume-polynomial} and the
definition of the dual polynomial give $
 f^\vee=|\det\mathcal V|$, 
so the conclusion follows.  We may therefore assume that $m>n$.

If $\operatorname{rank}\mathcal V<n$, then $f=f^\vee=0$, and the conclusion follows
by taking $m$ zero vectors in $\R^{m-n}$.  We may therefore assume that
$\mathcal V$ has rank $n$.

Put $r=m-n$.  After relabeling the columns, we may suppose that the first
$n$ columns form an invertible matrix $\mathcal V_0$.  Set
\[
 \mathcal A=\mathcal V_0^{-1}(v_{n+1}\ \cdots\ v_m),
 \qquad
 \widetilde{\mathcal V}=(I_n\ \mathcal A),
 \qquad
 \widetilde{\mathcal W}=(-\mathcal A^{\mathsf T}\ I_r).
\]
Thus $\mathcal V=\mathcal V_0\widetilde{\mathcal V}$ and
$\widetilde{\mathcal V}\widetilde{\mathcal W}^{\mathsf T}=0$;
equivalently, $\widetilde{\mathcal W}$ is a Gale dual matrix of
$\widetilde{\mathcal V}$.

We claim the maximal minors of these two matrices are
complementary. Indeed, given $I\subset[m]$ with $|I|=n$, there are sets
$R\subset[n]$ and $C\subset[r]$, with $|R|=|C|$, such that
\[
 I=([n]\setminus R)\cup\{n+j:j\in C\}.
\]
Expanding $\det\widetilde{\mathcal V}_I$ along its columns from $I_n$, and
expanding $\det\widetilde{\mathcal W}_{I^c}$ along its columns from $I_r$,
leaves in both cases the same minor $\det \mathcal A_{R,C}$, up to sign.
Hence
\begin{equation}\label{eq:gale-complementary-minors}
 |\det\widetilde{\mathcal W}_{I^c}|
 =|\det\widetilde{\mathcal V}_I|.
\end{equation}

Because $r\geq1$, choose an invertible $r\times r$ matrix $\mathcal T$ with
$|\det \mathcal T|=|\det \mathcal V_0|$, for example
$\mathcal T=\operatorname{diag}(|\det \mathcal V_0|,1,\ldots,1)$, and put
$\mathcal W=\mathcal T\widetilde{\mathcal W}$.
Equation~\eqref{eq:gale-complementary-minors} now gives
\[
 |\det \mathcal W_{I^c}|
 =|\det \mathcal V_0|\,|\det\widetilde{\mathcal V}_I|
 =|\det \mathcal V_I|.
\]
If $w_1,\ldots,w_m\in\R^r$ are the columns of $\mathcal W$, then, writing
$J=I^c$, we obtain
$$
 f^\vee(x)
 =\sum_{\substack{J\subset[m]\\|J|=r}}
   |\det \mathcal V_{J^c}|\,x^J
 =\sum_{\substack{J\subset[m]\\|J|=r}}
   |\det \mathcal W_J|\,x^J
 =\left|\sum_{j=1}^m x_j[0,w_j]\right|_r.
$$
This is exactly the volume polynomial of the indicated $m$ segments in
$\R^{m-n}$.
\end{proof}

The log-submodularity inequality for zonoids was proved in dimensions at most three in~\cite{FMMZ-24}.  We recall explicitly the polynomial consequence that is used here.

\begin{theorem}
\label{thm:dim3-rayleigh}
The zonotope volume polynomial of any finite collection of line segments in
$\R^n$ is Rayleigh whenever $1\leq n\leq3$.
\end{theorem}

\begin{proof}
For $n=1$ the volume polynomial is linear, so
$\partial_{ij}f=0$ for $i\neq j$ and the Rayleigh inequality is immediate.
Let $n\in\{2,3\}$. It is shown in \cite{FMMZ-24} that in this case, Conjecture  \ref{conj:Plun-Ruz_zonoids} is true and hence, by Proposition \ref{prop: equiv}, the claim follows.
\end{proof}

Combining the aforementioned results, we now prove the first half of Theorem~\ref{thm:rayleigh-threshold}.
\begin{proof}[Proof of the  case $\min (n,m-n)\leq 3$ in Theorem~\ref{thm:rayleigh-threshold}]
Let $f$ be the volume polynomial of $m$ segments in $\R^n$, and put
$r=m-n$.  If $n\leq3$, Theorem~\ref{thm:dim3-rayleigh} applies directly.
Suppose, therefore, that $n>3$.  Then the assumption
$\min\{n,m-n\}\leq3$ gives $0\leq r\leq3$.

By Lemma~\ref{lem:Gale}, $f^\vee$ is a nonnegative constant when $r=0$
and is a zonotope volume polynomial in $\R^r$ when $1\leq r\leq3$.  Thus
$f^\vee$ is Rayleigh, trivially in the first case and by
Theorem~\ref{thm:dim3-rayleigh} in the second.  Lemma~
\ref{lem:rayleigh-duality} now shows that $f$ is Rayleigh.
\end{proof}

\section{Counterexamples to zonoid volume and projection inequalities}\label{sec:zonoid-volume}

\subsection{A Rayleigh counterexample from a representable matroid}
\label{sec:rayleigh-counterexample}

\begin{proof}[Proof of the  case $\min (n,m-n)> 3$  in Theorem~\ref{thm:rayleigh-threshold}]
We begin with $n=4$ and $m=8$.  Consider the matrix
\[
{\mathcal A}=
\begin{pmatrix}
1&1&1&1&1&1&1&3\\
0&1&0&0&2&0&0&1\\
0&0&1&0&0&2&0&1\\
0&0&0&1&0&0&2&3
\end{pmatrix},
\]
and let $v_1,\ldots,v_8\in\R^4$ be its columns.  This matrix represents the
matroid $\mathcal J'$ from~\cite[Theorem~5.11]{CW}.  For
$x\in\R_+^8$, let $Z_x=\sum_{i=1}^8x_i[0,v_i]$.  By
\eqref{eq:zonotope-volume},
\begin{equation}\label{eq:volpoly}
 f(x)=|Z_x|
 =\sum_{\substack{S\subset[8]\\|S|=4}}
   |\det \mathcal A_S|x^S,
\end{equation}
where $\mathcal A_S$ is the submatrix with column set $S$.
 At
$
 a=(1,3,3,1,1,3,3,1)
$
a direct computation gives
$
 f(a)=1658,
 \partial_1f(a)=548,
 \partial_8f(a)=704,
 \partial_{18}f(a)=233.
$ 
Consequently,
\begin{equation}\label{eq:rayleigh-defect-example}
 \partial_1f(a)\partial_8f(a)-f(a)\partial_{18}f(a)=-522<0.
\end{equation}
Thus $f$ is not Rayleigh.

To see the corresponding failure of \eqref{eq:Plun-Ruz_zonoids} in  Conjecture~\ref{conj:Plun-Ruz_zonoids} directly, put
\[
 A=Z_a,\qquad B=[0,v_1],\qquad C=[0,v_8].
\]
Since $f$ is affine   in each variable separately, we have
\[
 |A+B|\,|A+C|-|A|\,|A+B+C|
 =\partial_1f(a)\partial_8f(a)-f(a)\partial_{18}f(a),
\]
which is negative by~\eqref{eq:rayleigh-defect-example}.

For $n>4$, set
\[
 \widetilde A=A\times[0,1]^{n-4},\qquad
 \widetilde B=B\times\{0\}^{n-4},\qquad
 \widetilde C=C\times\{0\}^{n-4}.
\]
Then $|A|=|\widetilde A|$, $|A+B| = |\widetilde A+\widetilde B|$ and so on, hence the inequality \eqref{eq:Plun-Ruz_zonoids}   also fails.

Note that the construction uses $n+4$ generators.  For any $m>n+4$, appending arbitrary nonzero vectors suffices to finish the proof. 
\end{proof}

The preceding example is conceptually useful because it comes directly from
a non-Rayleigh representable matroid.  We now turn to smaller and more
structured counterexamples to the projection formulation.

\subsection{An \texorpdfstring{$(n+2)$}{n+2}-generator counterexample in dimension \texorpdfstring{$n \geq 4$}{n >= 4}}
\label{sec:counter_example in R4}

We first prove the main assertion of Theorem~\ref{thm:equivlent-conj}
by giving a six-generator counterexample in $\R^4$. Let $Z=\sum_{i=1}^6[0,a_i]$ be the zonotope in $\R^4$ 
whose generators
are the columns of 
\[
\mathcal A=
\begin{pmatrix}
1&0&1&1&0&0\\
0&1&1&0&0&0\\
0&1&1&-1&-3&0\\
2&2&2&0&0&-1
\end{pmatrix}.
\]
Let $\mathcal A^{(1)}$, $\mathcal A^{(2)}$, and
$\mathcal A^{(12)}$ denote the matrices obtained from
$\mathcal A$ by deleting, respectively, the first row, the second row,
and the first two rows. For each of these matrices, let $N_j$ denote
the number of maximal minors having absolute value $j$. A direct
determinant computation gives
\[
\begin{array}{c|ccccc|c}
 & N_0 & N_1 & N_2 & N_3 & N_6
 & \displaystyle\sum_{j}jN_j\\ \hline
\mathcal A       & 3  & 3 & 1 & 5 & 3 & 38\\
\mathcal A^{(1)}  & 12 & 2 & 2 & 2 & 2 & 24\\
\mathcal A^{(2)}  & 3  & 5 & 4 & 3 & 5 & 52\\
\mathcal A^{(12)} & 3  & 3 & 5 & 1 & 3 & 34
\end{array}
\]
By the zonotope volume formula~\eqref{eq:zonotope-volume-formula}, deleting
the appropriate rows corresponds to taking the relevant coordinate
projections. Consequently,
\[
|Z|=38,\qquad
|P_{e_1^\perp}Z|_3=24,\qquad
|P_{e_2^\perp}Z|_3=52,
\qquad
|P_{\{e_1,e_2\}^\perp}Z|_2=34.
\]
It follows that
\[
\frac{|Z|\,|P_{\{e_1,e_2\}^\perp}Z|_2}
{|P_{e_1^\perp}Z|_3\,|P_{e_2^\perp}Z|_3}
=
\frac{38\cdot34}{24\cdot52}
=
\frac{323}{312}>1.
\]
Thus, $Z$ violates~\eqref{eq:equivlent-conj}.

For $n>4$, the product
$
Z\times[0,1]^{n-4}\subset\R^n
$
gives the same ratio for the directions $e_1$ and $e_2$. Hence,
\eqref{eq:equivlent-conj} fails in every dimension $n\geq4$.
The resulting zonotope has $n+2$ generators. This is exactly one
above the sharp threshold for part~\textup{(iii)} in
Theorem~\ref{thm:rayleigh-threshold}: the inequality holds for every zonotope
in $\R^n$ generated by at most $n+1$ segments, a result originally proved
in~\cite{ADS}.

\subsection{A unimodular counterexample in dimension \texorpdfstring{$n \geq 4$}{n >= 4}}

\leavevmode
The preceding generating matrix is not unimodular. We now give a counterexample
whose nonzero maximal minors all have absolute value $1$. This
example also illustrates why determinant weights at $p=1$ and $p=2$ can
behave differently: one of its projections contains exactly one heavy
tile.  This distinction is pursued further in the companion
paper \cite{FHMNWZ-26-II}.

Let
$
Z=\sum_{i=1}^7[0,a_i]
$ be the zonotope in $\R^4$,
whose generators
are the columns of
\[
\mathcal U=
\begin{pmatrix}
1&1&2&1&0&0&1\\
1&1&1&0&0&0&0\\
1&1&1&0&0&1&1\\
0&0&0&1&1&-1&-1
\end{pmatrix}.
\]
Notice that $a_1=a_2$. Let $\mathcal U^{(1)}$,
$\mathcal U^{(2)}$, and $\mathcal U^{(12)}$
denote the matrices obtained by deleting the indicated rows.
With $N_j$ having the same meaning as in
Section~\ref{sec:counter_example in R4}, direct enumeration of the maximal
minors gives
\[
\begin{array}{c|ccc|c}
 & N_0 & N_1 & N_2
 & \displaystyle\sum_j jN_j\\ \hline
\mathcal U       & 15 & 20 & 0 & 20\\
\mathcal U^{(1)}  & 23 & 12 & 0 & 12\\
\mathcal U^{(2)}  & 10 & 24 & 1 & 26\\
\mathcal U^{(12)} & 5  & 16 & 0 & 16
\end{array}
\]
In particular, every nonzero $4\times4$ minor of
$\mathcal U$ has absolute
value $1$, so the generating system is unimodular. The only projected
maximal minor having absolute value greater than $1$ occurs in
$\mathcal U^{(2)}$: the columns indexed by $\{3,5,6\}$ have determinant of
absolute value $2$. Geometrically, this is the unique fat tile in
$P_{e_2^\perp}Z$; every other nonzero tile appearing in the four
determinant sums has volume $1$.

We emphasize that unimodularity here refers to the displayed
seven-column generating matrix of $Z$.  After adjoining the two projection
directions $e_1,e_2$, the full generator matrix is no longer unimodular: the
minor of absolute value $2$ just identified becomes a maximal minor of the
augmented matrix.  Thus this example is consistent with the positive result
of Skorupinski~\cite{S-26}, which assumes that the full combined generator
matrix is unimodular.

The table therefore gives
\[
|Z|=20,\qquad
|P_{e_1^\perp}Z|_3=12,\qquad
|P_{e_2^\perp}Z|_3=26,
\qquad
|P_{\{e_1,e_2\}^\perp}Z|_2=16.
\]
Consequently,
\[
\frac{|Z|\,|P_{\{e_1,e_2\}^\perp}Z|_2}
{|P_{e_1^\perp}Z|_3\,|P_{e_2^\perp}Z|_3}
=
\frac{20\cdot16}{12\cdot26}
=
\frac{40}{39}>1.
\]
Thus, even unimodularity of the generating system does not imply the
projection inequality.  For $n>4$, the product construction at
the end of Section~\ref{sec:counter_example in R4} gives a counterexample in
$\R^n$.

We also note that this example is not totally unimodular, as is
already clear from the entry $2$ in $\mathcal U$.

\subsection{A sharp graphical totally unimodular counterexample}

For a totally unimodular generating matrix,  inequality
\eqref{eq:equivlent-conj} holds when the projection directions are coordinate vectors: deleting rows
preserves total unimodularity, and the corresponding $L_1$ determinant sums
agree with their $L_2$ analogues.  The following example shows that the
inequality can nevertheless fail for non-coordinate orthonormal directions,
even for a graphical zonotope.

Let $H$ be the graph with vertex set $
V(H)=\{0,1,2,3,4,5\}$ 
and edge set $
E(H)=\{03,04,05,12,13,15,24\}.$
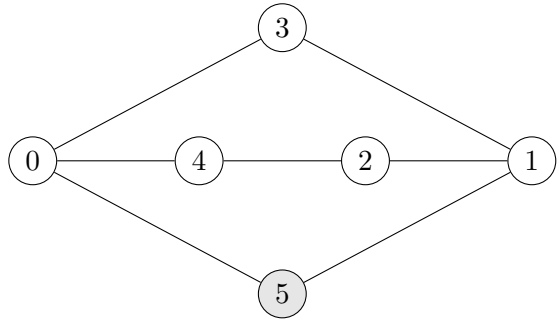
\begin{center}
\begin{tikzpicture}[
    scale=1.1,
    every node/.style={circle,draw,inner sep=1.5pt,minimum size=18pt}]
  \node (0) at (-3,0)     {$0$};
  \node (1) at (3,0)      {$1$};
  \node (2) at (1,0)      {$2$};
  \node (3) at (0,1.6)    {$3$};
  \node (4) at (-1,0)     {$4$};
  \node[fill=gray!20] (5) at (0,-1.6) {$5$};

  \draw (0)--(3) (0)--(4) (0)--(5)
        (1)--(2) (1)--(3) (1)--(5) (2)--(4);
\end{tikzpicture}
\end{center}
Choose the shaded vertex $5$ as the root, orient each edge toward its
smaller endpoint, and delete the incidence row corresponding to the root.
With rows indexed by $0,1,2,3,4$ and columns ordered as $
03,\ 04,\ 05,\ 12,\ 13,\ 15,\ 24,
$
the resulting reduced signed incidence matrix is
\[
\mathcal A_H=
\begin{pmatrix}
 1& 1&1& 0& 0&0& 0\\
 0& 0&0& 1& 1&1& 0\\
 0& 0&0&-1& 0&0& 1\\
-1& 0&0& 0&-1&0& 0\\
 0&-1&0& 0& 0&0&-1
\end{pmatrix}.
\]
Thus $\mathcal A_H$ is totally unimodular.  Since $H$ is connected,
$\mathcal A_H$ has rank five.  Denote its columns by $a_e$ and let $
Z=\sum_{e\in E(H)}[0,a_e]\subset\R^5.$ 

Set
\[
x=(1,-1,0,0,0),
\qquad
y=(0,0,1,-1,1),
\qquad
u=\frac{x}{\sqrt2},
\qquad
v=\frac{y}{\sqrt3}.
\]
Then $u$ and $v$ are orthonormal.  For $F\subseteq E(H)$, let
$\mathcal A_{H,F}$ denote the submatrix formed by the columns indexed by
$F$.  An exact enumeration of the relevant $5\times5$ determinants is
given below.  In each row, $N_j$ is the number of determinants having
absolute value $j$.
\[
\begin{array}{c|ccc|c}
\text{determinant sum} & N_0 & N_1 & N_2
    & \text{value of the sum}\\ \hline
\displaystyle\sum_{|F|=5}|\det(\mathcal A_{H,F})|
    &5&16&0&16\\[2mm]
\displaystyle\sum_{|F|=4}|\det(x,\mathcal A_{H,F})|
    &23&12&0&12\\[2mm]
\displaystyle\sum_{|F|=4}|\det(y,\mathcal A_{H,F})|
    &5&24&6&36\\[2mm]
\displaystyle\sum_{|F|=3}|\det(x,y,\mathcal A_{H,F})|
    &11&20&4&28
\end{array}
\]
As a check on the first row, $H$ is a theta graph whose three paths between
$0$ and $1$ have lengths $2,2,3$; hence it has
$2\cdot2+2\cdot3+2\cdot3=16$ spanning trees.
By the zonotope volume formula \eqref{eq:zonotope-volume-formula} and homogeneity, these sums give
\[
|Z|=16,
\qquad
|P_{u^\perp}Z|_4=\frac{12}{\sqrt2},
\qquad
|P_{v^\perp}Z|_4=\frac{36}{\sqrt3},
\qquad
|P_{\{u,v\}^\perp}Z|_3=\frac{28}{\sqrt6}.
\]
Consequently,
\[
\frac{|Z|\,|P_{\{u,v\}^\perp}Z|_3}
{|P_{u^\perp}Z|_4\,|P_{v^\perp}Z|_4}
=
\frac{16\cdot28}{12\cdot36}
=
\frac{28}{27}>1.
\]
Thus the projection inequality fails for a seven-generator graphical
zonotope with a totally unimodular generating matrix.  Notice that the
augmented matrix $(\mathcal A_H,x,y)$ is not totally unimodular, as witnessed
by the determinants of absolute value $2$ in the table.  This example has
$n+2=7$ generators, which is the first possible number by
Theorem~\ref{thm:rayleigh-threshold}.  For $n>5$, attach $n-5$ pendant edges
to the root of $H$; equivalently, take the product of $Z$ with
$[0,1]^{n-5}$.  This preserves the ratio $28/27$ and gives a graphical
totally unimodular counterexample in $\R^n$ with $n+2$ generators.

\subsection{Failure of the strong projection inequality}
\label{sec:str_conj}

\begin{proof}[Proof of Theorem~\ref{thm:str_conj}]
Let $A$ and $B$ be the zonotopes generated, respectively, by the columns of
the following matrices:
\begin{equation}
    \mathcal U =\begin{pmatrix}
        1&2&2&1\\
        0&0&-2&-1\\
        3&-1&0&3\\
    \end{pmatrix} 
    \quad \text{and} \quad
    \mathcal V=\begin{pmatrix}
        2&1&0&0\\
        0&0&1&2\\
        -1&3&3&-1\\
    \end{pmatrix} .
\end{equation}
Using \eqref{eq:zonotope-volume-formula}, we obtain
$|A|=39$, $|P_{e_3^\perp}A|_2=9$, $|B|=42$, and
$|P_{e_3^\perp}B|_2=9$.  Since $A+B$ is generated by the columns of the
concatenated matrix
\begin{equation}\label{eq:R4-G}
    \mathcal W=
    (\mathcal U\ \mathcal V)
    =\begin{pmatrix}
        1&2&2&1&2&1&0&0\\
        0&0&-2&-1&0&0&1&2\\
        3&-1&0&3&-1&3&3&-1\\
    \end{pmatrix} ,
\end{equation}
we also have $|A+B|=399$ and $|P_{e_3^\perp}(A+B)|_2=45$.  Therefore,
\[
\frac{|A+B|}{|P_{e_3^\perp}(A+B)|_2}
=\frac{399}{45}<\frac{405}{45}
=\frac{|A|}{|P_{e_3^\perp}A|_2}
+\frac{|B|}{|P_{e_3^\perp}B|_2}.
\]
This disproves Conjecture~\ref{conj:str_conj} in dimension three.

For completeness, let $n>3$ and set
\[
 \widetilde A=A\times[0,1]^{n-3},
 \qquad
 \widetilde B=B\times[0,1]^{n-3}.
\]
For $\widetilde u=(e_3,0)\in\R^3\times\R^{n-3}$, the two individual
ratios are unchanged:
\[
 \frac{|\widetilde A|}{|P_{\widetilde u^\perp}\widetilde A|_{n-1}}
 =\frac{|A|}{|P_{e_3^\perp}A|_2},
 \qquad
 \frac{|\widetilde B|}{|P_{\widetilde u^\perp}\widetilde B|_{n-1}}
 =\frac{|B|}{|P_{e_3^\perp}B|_2}.
\]
Moreover,
$\widetilde A+\widetilde B=(A+B)\times[0,2]^{n-3}$, so the common factor
$2^{n-3}$ cancels from the corresponding ratio for the sum.  Hence, the
same strict reverse inequality holds in every dimension $n>3$.
\end{proof}

\section{Courtade's conjecture and its failure}
\label{sec:courtade}
We refer to~\cite{Schneider-14} for the standard background used in this section. The support function of $K$ is \(h_K(x)=\sup_{y\in K}\langle x,y\rangle\), for \(x\in\R^n\). 
We write $\V(K_1,\ldots,K_n)$ for mixed volume and use $K[j]$ for $j$ copies
of $K$. We begin by proving the common
necessary surface-area condition stated in the introduction.  Recall that
$|\partial K|_{n-1}=n\V(K[n-1],B_2^n)$.

\begin{prop}\label{prop:courtade-surface}
Suppose that Courtade's inequality
\begin{equation}\label{eq:courtade}
 (|B|\,|C|)^{1/n}+(|B_2^n|\,|B_2^n+B+C|)^{1/n}
 \leq |B_2^n+B|^{1/n}|B_2^n+C|^{1/n}
\end{equation}
holds for any convex bodies $B,C \subset \R^n$.
Then, the surface-area inequality 
\begin{equation}\label{eq:surface-obstruction}
 \left(\frac{|\partial B|_{n-1}}{|B|}+\frac{|\partial C|_{n-1}}{|C|}\right)^n
 \geq n^n|B_2^n|\frac{|B+C|}{|B|\,|C|}
\end{equation}
holds for any full-dimensional convex bodies $B,C \subset \R^n$. 
Thus, every counterexample $B,C$ in $\mathbb R^n$ to 
\eqref{eq:surface-obstruction} gives, after a sufficiently large common dilation,  a counterexample to \eqref{eq:courtade}.
\end{prop}

\begin{proof}
Let $K$ be a full-dimensional convex body. By Steiner expansion,  as $ t \to \infty $,
\[
 |tK+B_2^n|=t^n|K|+t^{n-1}|\partial K|_{n-1}+O(t^{n-2}) =t^n|K|\left(1+\frac{|\partial K|_{n-1}}{t|K|}+O\left(t^{-2}\right)\right).
\]
Taking $n$th roots 
and using Taylor expansion
$
(1+x)^{1 / n}=1+\frac{x}{n}+O\left(x^2\right) 
$
as $x \to 0$, gives,
\begin{equation}\label{eq:courtade-root-expansion}
 |tK+B_2^n|^{1/n}
 =t|K|^{1/n}
 +\frac{|\partial K|_{n-1}}{n|K|^{(n-1)/n}}+O(t^{-1}), \qquad t \to \infty.
\end{equation}
Now, fix full-dimensional convex bodies $B,C \subset \R^n$.
We apply \eqref{eq:courtade} to $tB,tC$:
\begin{equation} \label{eq:courtade-tB-tC}
    (|t B||t C|)^{1 / n}+\left(\left|B_2^n\right|\left|B_2^n+t B+t C\right|\right)^{1 / n} \leq\left|B_2^n+t B\right|^{1 / n}\left|B_2^n+t C\right|^{1 / n} .
\end{equation}
Using \eqref{eq:courtade-root-expansion}, the right-hand side of \eqref{eq:courtade-tB-tC} is equal to
\begin{align*}
 |B_2^n+tB|^{1/n}|B_2^n+tC|^{1/n}
 ={}&t^2(|B|\,|C|)^{1/n}
 +\frac{t(|B|\,|C|)^{1/n}}{n}
 \left(\frac{|\partial B|_{n-1}}{|B|}+\frac{|\partial C|_{n-1}}{|C|}\right)+O(1).
\end{align*}
Also, the left-hand side of \eqref{eq:courtade-tB-tC} is equal to 
$
 t^2(|B|\,|C|)^{1/n}
+
 t(|B_2^n||B+C|)^{1/n}+O(1).
$
Canceling the common leading term, dividing by $t$, and then letting $ t \to \infty$, we obtain
\[
 (|B_2^n||B+C|)^{1/n}
 \leq\frac{(|B|\,|C|)^{1/n}}{n}
 \left(\frac{|\partial B|_{n-1}}{|B|}+\frac{|\partial C|_{n-1}}{|C|}\right).
\]
Raising this inequality to the $n$th power proves
\eqref{eq:surface-obstruction}.  
\end{proof}
It is convenient to use the following quantity 
\begin{equation}\label{eq:courtade-Q}
 \mathcal Q_n(B,C)=
 \frac{
 \left(\dfrac{|\partial B|_{n-1}}{|B|}+\dfrac{|\partial C|_{n-1}}{|C|}\right)^n
 |B|\,|C|}
 {n^n|B_2^n||B+C|}.
\end{equation}
Thus, $\mathcal Q_n(B,C)<1$ is  the strict reverse of
\eqref{eq:surface-obstruction}, moreover, $\mathcal Q_n(B,C)$  is invariant under a common dilation of bodies $B$ and $C$.
\subsection{A counterexample to Courtade's conjecture}
\label{sec:courtade-double-bodies of revolution}

In dimension three, there is a simple geometric counterexample to Conjecture~\ref{conj:courtade} among centrally symmetric convex bodies.  As before, let  $e_1,e_2,e_3$ be the standard
basis of $\R^3$, and set
\begin{equation}\label{eq:courtade-bodies of revolution}
 K_1=\operatorname{conv}\bigl(B_2^3,\{-3e_1,3e_1\}\bigr),
 \qquad
 K_2=\operatorname{conv}\bigl(B_2^3,\{-3e_2,3e_2\}\bigr).
\end{equation}
These are centrally symmetric double bodies of revolution with orthogonal
axes; see Figure~\ref{fig:courtade-double-bodies of revolution}.
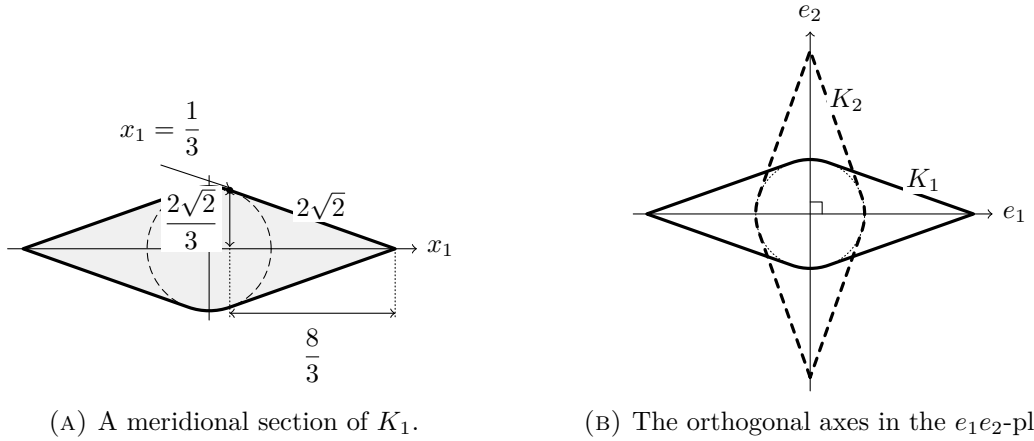
\begin{figure}[h]
\centering
\begin{subfigure}[t]{0.53\textwidth}
\centering
\begin{tikzpicture}[
  x=0.82cm,y=0.82cm,
  line cap=round,line join=round,
  every node/.style={font=\small}]
\def\rr{0.942809}
\coordinate (O)  at (0,0);
\coordinate (T)  at (3,0);
\coordinate (TR) at (0.333333,\rr);

\path[fill=black!6]
  (-3,0)--(-0.333333,\rr)
  arc[start angle=109.4712,end angle=70.5288,radius=1]
  --(3,0)--(0.333333,-\rr)
  arc[start angle=-70.5288,end angle=-109.4712,radius=1]--cycle;

\draw[thin,densely dashed] (O) circle[radius=1];
\draw[->,thin] (-3.25,0)--(3.35,0) node[right] {$x_1$};
\draw[thin] (0,-1.15)--(0,1.18);

\draw[very thick]
  (-3,0)--(-0.333333,\rr)
  arc[start angle=109.4712,end angle=70.5288,radius=1]
  --(3,0)--(0.333333,-\rr)
  arc[start angle=-70.5288,end angle=-109.4712,radius=1]--cycle;

\draw[<->] (0.333333,0.03)--(TR)
  node[midway,left=3pt,fill=white,inner sep=1pt]
       {$\dfrac{2\sqrt2}{3}$};
\fill (TR) circle[radius=1.3pt];

\node[anchor=south east] (tangencylabel) at (0.05,1.35)
  {$x_1=\dfrac13$};
\draw[->,thin] (tangencylabel.south)--($(TR)+(0,0.04)$);

\node[above,sloped,fill=white,inner sep=1pt]
  at ($(TR)!0.55!(T)$) {$2\sqrt2$};

\draw[densely dotted] (0.333333,0)--(0.333333,-1.10);
\draw[densely dotted] (T)--(3,-1.10);
\draw[<->] (0.333333,-1.04)--(3,-1.04)
  node[midway,below=2pt] {$\dfrac83$};
\end{tikzpicture}
\subcaption{A meridional section of $K_1$.}
\label{fig:courtade-spindle-meridian}
\end{subfigure}\hfill
\begin{subfigure}[t]{0.43\textwidth}
\centering
\begin{tikzpicture}[
  x=0.72cm,y=0.72cm,
  line cap=round,line join=round,
  every node/.style={font=\small}]
\def\rr{0.942809}

\draw[->,thin] (-3.25,0)--(3.35,0) node[right] {$e_1$};
\draw[->,thin] (0,-3.25)--(0,3.35) node[above] {$e_2$};
\draw[thin,densely dotted] (0,0) circle[radius=1];

\draw[very thick]
  (-3,0)--(-0.333333,\rr)
  arc[start angle=109.4712,end angle=70.5288,radius=1]
  --(3,0)--(0.333333,-\rr)
  arc[start angle=-70.5288,end angle=-109.4712,radius=1]--cycle;

\begin{scope}[rotate=90]
\draw[very thick,dashed]
  (-3,0)--(-0.333333,\rr)
  arc[start angle=109.4712,end angle=70.5288,radius=1]
  --(3,0)--(0.333333,-\rr)
  arc[start angle=-70.5288,end angle=-109.4712,radius=1]--cycle;
\end{scope}

\draw (0.22,0)--(0.22,0.22)--(0,0.22);
\node[fill=white,inner sep=1pt] at (2.08,0.62) {$K_1$};
\node[fill=white,inner sep=1pt] at (0.67,2.10) {$K_2$};
\end{tikzpicture}
\subcaption{The orthogonal axes in the $e_1e_2$-plane.}
\label{fig:courtade-spindle-axes}
\end{subfigure}

\caption{In~\textup{(a)}, the dashed circle
is the meridional section of $B_2^3$; the labels give the radius of the
tangency circle, the cone height, and the slant height. In~\textup{(b)},
the solid and dashed curves are the sections of $K_1$ and $K_2$,
respectively, by $\operatorname{span}\{e_1,e_2\}$.}
\label{fig:courtade-double-bodies of revolution}
\end{figure}

\begin{prop}
\label{prop:courtade-bodies of revolution}
The bodies $K_1,K_2$ satisfy $\mathcal Q_3(K_1,K_2)<1$.  Thus,
the pair $tK_1,tK_2$ violates \eqref{eq:courtade} for every sufficiently
large $t$.  
\end{prop}

\begin{proof}
 For $K_1$, the two tangency circles
between the unit sphere and the conical parts of the boundary lie in the
planes $x_1=\pm1/3$.  Their radius is $2\sqrt2/3$; each cone has height
$8/3$ and slant height $2\sqrt2$.  Thus, the central spherical zone and the
two cones give
\begin{align}
 |K_1|
 &=\pi\int_{-1/3}^{1/3}(1-x^2)\,dx
   +2\cdot\frac13\pi\left(\frac{2\sqrt2}{3}\right)^2\frac83
   =\frac{20\pi}{9},
 \label{eq:courtade-spindle-volume}\\
 |\partial K_1|_2
 &=\frac{4\pi}{3}
   +2\pi\frac{2\sqrt2}{3}\,2\sqrt2
   =\frac{20\pi}{3}.
 \label{eq:courtade-spindle-surface}
\end{align}
The same formulas hold for $K_2$, so both surface-to-volume ratios equal
$3$.  Since $|B_2^3|=4\pi/3$, the necessary condition
\eqref{eq:surface-obstruction} for this pair is equivalent to
\[
 |K_1+K_2|\leq\frac{800\pi}{27}.
\]
We prove the strict reverse inequality.  The rotation interchanging $e_1$ and $e_2$ interchanges $K_1$ and $K_2$.
Hence
$\V(K_1[2],K_2)=\V(K_1,K_2[2])$. Additionally, let $S_{K_1}$ be the surface-area
measure of $K_1$ on $\s^2$.  Using that 
$
 \V(K_1[2],K_2)=\frac{1}{3}\int_{\s^2}h_{K_2}(u)\,dS_{K_1}(u)
$ and the mixed-volume expansion yields
\begin{equation}\label{eq:courtade-spindle-sum-volume}
 |K_1+K_2|=2|K_1|+6V(K_1[2],K_2)=2|K_1|+2\int_{\s^2}h_{K_2}(u)\,dS_{K_1}(u).
\end{equation}

For $u\in\s^2$, $
 h_{K_2}(u)=\max\{1,3|u_2|\}$. 
On the belt, $S_{K_1}$ agrees with the spherical Hausdorff-area measure
$\mathcal H^2|_E$, where $
 E=\{u\in\s^2:|u_1|\leq1/3\}.$
The two conical parts contribute uniform measures, each of total mass
$8\pi/3$, on the normal circles
$
 \Gamma_\pm=\{u\in\s^2:u_1=\pm1/3\}.
$

On the spherical belt, parametrization by $u_1=x$ gives
\begin{equation}\label{eq:courtade-spindle-belt}
 \int_E h_{K_2}(u)\,d\mathcal H^2(u)
 \geq3\int_E|u_2|\,d\mathcal H^2(u) 
 =12\int_{-1/3}^{1/3}\sqrt{1-x^2}\,dx
 =\frac{8\sqrt2}{3}+12\arcsin\frac13.
\end{equation}
On either normal circle write
$u_2=(2\sqrt2/3)\cos\varphi$.  For $a\geq1$ the elementary identity
\begin{equation}\label{eq:courtade-cosine-integral}
 \int_0^{2\pi}\max\{1,a|\cos\varphi|\}\,d\varphi
 =4\left(\sqrt{a^2-1}+\arcsin\frac1a\right)
\end{equation}
follows by splitting each quadrant where $a\cos\varphi=1$.  Since the
density of each uniform circle measure is $(8\pi/3)/(2\pi)=4/3$, the two
normal circles contribute
\begin{equation}\label{eq:courtade-spindle-cones}
 \int_{\Gamma_+\cup\Gamma_-}
 h_{K_2}(u)\,dS_{K_1}(u)
 =\frac{32}{3}\left(\sqrt7+
       \arcsin\frac1{2\sqrt2}\right).
\end{equation}
Combining \eqref{eq:courtade-spindle-belt} and
\eqref{eq:courtade-spindle-cones}, and using $\arcsin x>x$ for $0<x<1$,
$\sqrt2>7/5$, $\sqrt7>37/14$, and $\pi<22/7$,
we obtain
$$
 \int_{\s^2}h_{K_2}(u)\,dS_{K_1}(u)
  >\frac{340\pi}{27}.
$$
 Equations
\eqref{eq:courtade-spindle-volume} and
\eqref{eq:courtade-spindle-sum-volume} now give
\begin{equation}\label{eq:courtade-spindle-sum-lower}
 |K_1+K_2|
 >\frac{40\pi}{9}+\frac{680\pi}{27}
 =\frac{800\pi}{27}.
\end{equation}
Thus \eqref{eq:surface-obstruction} is strictly reversed, equivalently
$\mathcal Q_3(K_1,K_2)<1$.
\end{proof}

\begin{remark}
We note that neither $K_1$ nor $K_2$ is a zonoid. Indeed, the support function of every
zonoid satisfies Hlawka's inequality (see \cite{Witsenhausen-1978} or \cite{Schneider-Weil-1983})
\[
h_K(x)+h_K(y)+h_K(z)+h_K(x+y+z)
\geq
h_K(x+y)+h_K(x+z)+h_K(y+z).
\]
To see this, apply the corresponding scalar inequality for the absolute
value to $
\langle x,u\rangle,$  $
\langle y,u\rangle,$  $
\langle z,u\rangle
$
and integrate against a positive generating measure.  Recall that $
h_{K_2}(u)=\max\bigl\{|u|,3|u_2|\bigr\}.$ 
Choose
\[
x=2e_1+e_2+2e_3,\qquad
y=2e_1-e_2-2e_3,\qquad
z=-2e_1+e_2-2e_3.
\]
Each of $x,y,z,x+y+z$ has Euclidean norm $3$ and second coordinate of
absolute value $1$. Hence
$
h_{K_2}(x)=h_{K_2}(y)=h_{K_2}(z)
=h_{K_2}(x+y+z)=3.
$
On the other hand, $
x+y=4e_1, x+z=2e_2, y+z=-4e_3, $
and therefore
\[
h_{K_2}(x+y)+h_{K_2}(x+z)+h_{K_2}(y+z)
=4+6+4=14.
\]
Thus Hlawka's inequality would require $12\geq14$, a contradiction.
Consequently, $K_2$ is not a zonoid, and the same holds for its rotation
$K_1$.
\end{remark}

\subsection{Counterexamples among zonoids}
\label{sec:courtade-zonoids}

The fact that inequality \eqref{eq:Plun-Ruz_zonoids} fails for general convex bodies but holds for zonoids in $\mathbb{R}^3$ naturally suggests that \eqref{eq:courtade-intro} might hold whenever $B$ and $C$ are zonoids. In this section, we construct a counterexample to this conjecture.

\begin{theorem}\label{thm:courtade-counterexample}
Courtade's conjecture is false for zonoids in every dimension $n\geq3$.
More precisely:
\begin{enumerate}[label=\textup{(\roman*)}]
\item in $\R^3$ there exist an unconditional six-generator zonotope $B$ and a
      quarter-turn $R$ such that $tB,tRB$ violate \eqref{eq:courtade} for
      every sufficiently large $t$;
\item for every $n\geq4$ there are  zonoids
      $B,C\subset\R^n$  such that
      $B+C=2B_2^n$ and
      $tB,tC$ violate \eqref{eq:courtade} for every
      sufficiently large $t$.
\end{enumerate}
\end{theorem}

\subsubsection{A six-generator zonotope and its quarter-turn} In dimension three, we shall also use the standard
surface-area formula (see, for example \cite{Schneider-14}): $
 |\partial Z|_2=2\sum_{i<j}|z_i\times z_j|$ where $ Z=\sum_{i=1}^m[0,z_i].$ 
For \(\varepsilon,\delta\in\{-1,1\}\), set
$
 v_{\varepsilon,\delta}=(6,2\varepsilon,2\delta)$  and 
 $w_\delta=(5,0,3\delta).$
Define the centered zonotope
\[
 B=
 \sum_{\varepsilon,\delta\in\{-1,1\}}
 \left[-\frac{v_{\varepsilon,\delta}}2,
        \frac{v_{\varepsilon,\delta}}2\right]
 +
 \sum_{\delta\in\{-1,1\}}
 \left[-\frac{w_\delta}2,\frac{w_\delta}2\right].
\]
Let
\[
 R(x,y,z)=(-y,x,z)
\]
be the quarter-turn about the \(e_3\)-axis, and put \(C=RB\).  Reflections in
the coordinate hyperplanes permute the six unoriented generators of \(B\),
so \(B\) is unconditional. The six
centered generating segments and their images are shown in
Figure~\ref{fig:courtade-quarter-turn-generators}.
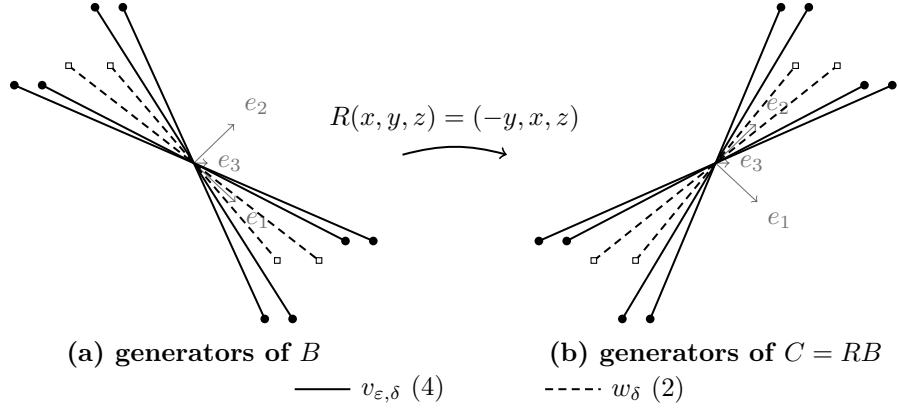
\begin{figure}[h]
\centering
\begin{tikzpicture}[
  scale=.92,
  axis/.style={->,thin,opacity=.55},
  vgen/.style={line width=.75pt},
  wgen/.style={line width=.75pt,densely dashed},
  endpoint/.style={circle,fill,inner sep=1.15pt},
  wendpoint/.style={rectangle,draw,fill=white,inner sep=1.05pt},
  every node/.style={font=\small}]

\begin{scope}[
  x={(.60cm,-.56cm)},
  y={(.58cm,.56cm)},
  z={(.20cm,0cm)}]

  \foreach \eps in {-1,1}{
    \foreach \del in {-1,1}{
      \draw[vgen]
        (-3,{-\eps},{-\del})--(3,\eps,\del);
      \node[endpoint] at (3,\eps,\del) {};
      \node[endpoint] at (-3,{-\eps},{-\del}) {};
    }
  }

  \foreach \del in {-1,1}{
    \draw[wgen]
      (-2.5,0,{-1.5*\del})--(2.5,0,{1.5*\del});
    \node[wendpoint] at (2.5,0,{1.5*\del}) {};
    \node[wendpoint] at (-2.5,0,{-1.5*\del}) {};
  }

  \draw[axis] (0,0,0)--(1,0,0)
    node[below right] {$e_1$};
  \draw[axis] (0,0,0)--(0,1,0)
    node[above right] {$e_2$};
  \draw[axis] (0,0,0)--(0,0,1)
    node[right] {$e_3$};
\end{scope}

\begin{scope}[
  xshift=7.5cm,
  x={(.60cm,-.56cm)},
  y={(.58cm,.56cm)},
  z={(.20cm,0cm)}]

  \foreach \eps in {-1,1}{
    \foreach \del in {-1,1}{
      \draw[vgen]
        (\eps,-3,{-\del})--({-\eps},3,\del);
      \node[endpoint] at ({-\eps},3,\del) {};
      \node[endpoint] at (\eps,-3,{-\del}) {};
    }
  }

  \foreach \del in {-1,1}{
    \draw[wgen]
      (0,-2.5,{-1.5*\del})--(0,2.5,{1.5*\del});
    \node[wendpoint] at (0,2.5,{1.5*\del}) {};
    \node[wendpoint] at (0,-2.5,{-1.5*\del}) {};
  }

  \draw[axis] (0,0,0)--(1,0,0)
    node[below right] {$e_1$};
  \draw[axis] (0,0,0)--(0,1,0)
    node[above right] {$e_2$};
  \draw[axis] (0,0,0)--(0,0,1)
    node[right] {$e_3$};
\end{scope}

\draw[->,line width=.75pt]
  (3.00,.12) to[bend left=15] (4.50,.12);
\node[font=\small] at (3.75,.65)
  {$R(x,y,z)=(-y,x,z)$};

\node[font=\small\bfseries] at (0,-2.75)
  {(a) generators of $B$};
\node[font=\small\bfseries] at (7.5,-2.75)
  {(b) generators of $C=RB$};

\draw[vgen] (1.45,-3.25)--(2.25,-3.25)
  node[right,font=\small]
  {$v_{\varepsilon,\delta}$ $(4)$};
\draw[wgen] (5.05,-3.25)--(5.85,-3.25)
  node[right,font=\small]
  {$w_\delta$ $(2)$};

\end{tikzpicture}

\caption{The six centered generating segments of $B$ and their images
under the quarter-turn $R$, shown in the same oblique projection.
The four $v_{\varepsilon,\delta}$-segments are solid, with circular
endpoints, whereas the two $w_\delta$-segments are dashed, with square
endpoints. The coordinate arrows run from the origin to
$e_1,e_2,e_3$; their unequal apparent lengths are due to the
projection.  The bundle
directed along $e_1$ is carried to the corresponding bundle directed
along $e_2$ for $C=RB$.}
\label{fig:courtade-quarter-turn-generators}
\end{figure}
Moreover,
 $R^2v_{\varepsilon,\delta}=-v_{\varepsilon,-\delta}$ and $
 R^2w_\delta=-w_{-\delta},
$
and hence \(R^2B=B\).  Thus \(R\) interchanges \(B\) and \(C\), which have
the same volume and surface area.  For this congruent pair,
\eqref{eq:surface-obstruction} is equivalent to
\begin{equation}\label{eq:courtade-quarter-turn-obstruction}
 2|\partial B|_2^3\geq 9\pi |B|\,|B+RB|.
\end{equation}
Let us order  the generators by
$
 (b_1,\ldots,b_6)=
(v_{-1,-1},v_{-1,1},v_{1,-1},v_{1,1},w_{-1},w_1).
$
The twenty absolute triple determinants needed for our computation are given by
\[
\begin{array}{c|l}
\left|\det(b_i,b_j,b_k)\right|
   & \text{triples }(i,j,k) \\ \hline
32  & (1,3,5),\ (2,4,6)\\
40  & (1,2,5),\ (1,2,6),\ (3,4,5),\ (3,4,6)\\
60  & (1,5,6),\ (2,5,6),\ (3,5,6),\ (4,5,6)\\
72  & (1,4,5),\ (1,4,6),\ (2,3,5),\ (2,3,6)\\
96  & (1,2,3),\ (1,2,4),\ (1,3,4),\ (2,3,4)\\
112 & (1,3,6),\ (2,4,5)
\end{array}
\]
It follows from \eqref{eq:zonotope-volume-formula} that
\[
 |B|=2(32)+4(40)+4(60)+4(72)+4(96)+2(112)=1360.
\]
Now we need to compute the surface area and the mixed terms. For later, define 
\[
 H(x)=\sum_{k=1}^6|\langle x,Rb_k\rangle|.
\]
Then, for \(i<j\),
\[
 H(b_i\times b_j)=\sum_{k=1}^6|\det(b_i,b_j,Rb_k)|.
\]
All fifteen pairs fall into the following six types; the symmetries of the generating set
make both \(|x|^2\) and \(H(x)\) constant within each type:
\begin{center}
\small
\begin{tabular}{c|c|c|r|r}
\toprule
pair type & multiplicity & \(x=b_i\times b_j\)
& \(|x|^2\) & \(H(x)\)\\
\midrule
same \(\varepsilon\) among the \(v\)'s
 &2&\((-8,-24,0)\)&640&816\\
same \(\delta\) among the \(v\)'s
 &2&\((8,0,24)\)&640&336\\
opposite in both signs
 &2&\((0,-24,24)\)&1152&816\\
\(v_{\varepsilon,\delta},w_\delta\)
 &4&\((6,8,10)\)&200&272\\
\(v_{\varepsilon,\delta},w_{-\delta}\)
 &4&\((-6,-28,10)\)&920&952\\
\(w_{-1},w_1\)
 &1&\((0,-30,0)\)&900&1020\\
\bottomrule
\end{tabular}
\end{center}
The multiplicity and squared-norm columns, together with the surface-area
formula, give
\begin{align*}
 |\partial B|_2
 &=2\bigl(4\sqrt{200}+4\sqrt{640}+\sqrt{900}
          +4\sqrt{920}+2\sqrt{1152}\bigr)\\
 &=60+64\sqrt{10}+176\sqrt2+16\sqrt{230}.
\end{align*}
To compute the mixed term, we use the last column. We note that
\[
 Rv_{\varepsilon,\delta}=(-2\varepsilon,6,2\delta),
 \qquad Rw_\delta=(0,5,3\delta).
\]
For example,
\[
 H(-8,-24,0)
 =2|16-144|+2|-16-144|+2|-120|=816.
\]
The remaining entries are checked in the same way.  Thus the
table accounts for all ninety mixed absolute determinants, grouped into six
symmetry classes, and their sum is
\[
\begin{split}
 M&=\sum_{1\leq i<j\leq6}\sum_{k=1}^6
 |\det(b_i,b_j,Rb_k)|\\
 &=2(816)+2(336)+2(816)+4(272)+4(952)+1020
 =9852.
\end{split}
\]
The other mixed contribution has the same value.  Indeed, applying
\(R^{-1}\) to every vector preserves absolute determinants, and
\(R^{-1}B=R^3B=RB\) because \(R^2B=B\).  Applying
\eqref{eq:zonotope-volume-formula} to the twelve generators of \(B+RB\) therefore
gives
\[
 |B+RB|=|B|+M+M+|RB|=1360+9852+9852+1360=22424.
\]
It remains only to compare the two sides of
\eqref{eq:courtade-quarter-turn-obstruction}.  The bounds
\[
 \sqrt{10}<\frac{19}{6},\qquad
 \sqrt2<\frac{17}{12},\qquad
 \sqrt{230}<\frac{91}{6}
\]
are checked by squaring; the excesses are \(1/36\), \(1/144\), and \(1/36\),
respectively.  Hence
\[
 |\partial B|_2
 <60+\frac{608}{3}+\frac{748}{3}+\frac{728}{3}
 =\frac{2264}{3}<755.
\]
Using also \(\pi>157/50\), we have
\[
 100\cdot755^3
 =43\,036\,887\,500
 <43\,091\,752\,320
 =9\cdot157\cdot1360\cdot22424.
\]
Consequently,
\[
 2|\partial B|_2^3
 <9\pi(1360)(22424)
 =9\pi|B|\,|B+RB|.
\]
Thus \(\mathcal Q_3(B,RB)<1\), and
Proposition~\ref{prop:courtade-surface} shows that \(tB,tRB\) violate
Courtade's inequality for every sufficiently large \(t\).  This proves
Theorem~\ref{thm:courtade-counterexample}(i).

\subsubsection{Ball-preserving perturbations in dimensions $n\geq4$}
\label{sec-virtual}

Having proved Theorem~\ref{thm:courtade-counterexample}(i) by an explicit
three-dimensional zonotope, we now turn to dimensions $n\geq4$ and prove
part~(ii).  The construction is perturbative: we shall produce two
 zonoids whose sum remains equal to $2B_2^n$ but for which the
necessary surface-area inequality is strictly reversed.  We begin with a
general perturbation criterion.    Set $m_n=(2|B_2^{n-1}|)^{-1}$ and $
 d\mu_n=m_n\,d\mathcal H^{n-1}.
$
Cauchy's projection formula gives, for every $x\in\s^{n-1}$,
\[
 \frac12\int_{\s^{n-1}}|\langle x,u\rangle|
 \,d\mathcal H^{n-1}(u)=|B_2^{n-1}|.
\]
Consequently,
\[
 \int_{\s^{n-1}}|\langle x,u\rangle|\,d\mu_n(u)
 =2m_n|B_2^{n-1}|=1,
\]
so $\mu_n$ is a generating measure of the Euclidean ball $B_2^n$.
We refer to \cite[Section~3.5]{Schneider-14} for generating measures of
zonoids and to \cite{Schneider-14} for Cauchy's projection formula.

Let $
 d\eta=f\,d\mathcal H^{n-1}
$
be an even signed measure on $\s^{n-1}$, where
$f\in C^\infty(\s^{n-1})$, and define its cosine transform by
\[
 h_\eta(x)=\int_{\s^{n-1}}|\langle x,u\rangle|\,d\eta(u).
\]
The cosine transform maps smooth functions to smooth functions; see, for
example, \cite[Chapter~3]{Groemer-96}.  Thus, $h_\eta$ is smooth.  If
$f\not\equiv0$ and $
 |s|<m_n/\|f\|_\infty,$
then
\[
 \mu_n+s\eta=(m_n+sf)\mathcal H^{n-1}
\]
has a smooth strictly positive density.   Let $Z_s$ be a zonoid generated by this measure.  Its support function is
\[
 h_{Z_s}=1+s h_\eta.
\]
For $u\in\s^{n-1}$, set
\[
 A_\eta(u)=\nabla_{\s^{n-1}}^2h_\eta(u)
              +h_\eta(u)I_{T_u\s^{n-1}}.
\]
Here $\nabla_{\s^{n-1}}^2$ is the covariant Hessian of the round metric,
and $I_{T_u\s^{n-1}}$ is the identity on the tangent space
$T_u\s^{n-1}$.  Since $h_\eta$ is smooth and the sphere is compact,
$h_\eta$ and the operator norm of $A_\eta$ are bounded.  Hence, after
decreasing $|s|$ if necessary,
\[
 1+s h_\eta(u)>0
 \quad\text{and}\quad
 I_{T_u\s^{n-1}}+sA_\eta(u)>0, \mbox{ for all }u\in\s^{n-1}.
\]
Thus $Z_s$ has smooth boundary and positive Gauss curvature.  The
support-function formula for volume gives
\begin{equation}\label{eq:courtade-support-volume}
 p_\eta(s)=\frac{|Z_s|}{|B_2^n|}
 =\frac{1}{n|B_2^n|}\int_{\s^{n-1}}
 \bigl(1+s h_\eta(u)\bigr)\times
 \det\!\left[
 I_{T_u\s^{n-1}}+sA_\eta(u)
 \right]d\mathcal H^{n-1}(u).
\end{equation}
Indeed, one combines
$
 |K|=\frac1n\int_{\s^{n-1}}h_K\,dS_K
$ 
with the identity
\[
 dS_K(u)
 =\det\!\left(
 \nabla_{\s^{n-1}}^2h_K(u)+h_K(u)I_{T_u\s^{n-1}}
 \right) d\mathcal H^{n-1}(u)
\]
for a smooth strictly convex body; see \cite[Chapters~2 and 5]{Schneider-14}.
The determinant in \eqref{eq:courtade-support-volume} is taken on the
$(n-1)$-dimensional space $T_u\s^{n-1}$.  It is a polynomial of degree
at most $n-1$ in $s$, and the factor preceding it is linear in $s$.
Therefore, for all sufficiently small $|s|$, the volume ratio
$p_\eta(s)$ agrees with a polynomial of degree at most $n$.

The following perturbation lemma reduces the problem to the second- and
fourth-order coefficients of this polynomial.

\begin{lemma}\label{lem:courtade-perturbation}
Let $n\geq4$, and let $d\eta=f\,d\mathcal H^{n-1}$ be a smooth even
signed measure on $\s^{n-1}$ such that
$\eta(\s^{n-1})=0$.  For all sufficiently small $|s|$, let $Z_s$ be the
zonoid generated by $\mu_n+s\eta$, and set
$p_\eta(s)=|Z_s|/|B_2^n|$.  Then, the linear coefficient of $p_\eta$
vanishes, and hence
\begin{equation}\label{eq:courtade-P-short}
 p_\eta(s)=1+a_2s^2+a_3s^3+a_4s^4+O(s^5).
\end{equation}
If
\begin{equation}\label{eq:courtade-fourth-short}
 \frac{n-2}{n}a_2^2-2a_4<0,
\end{equation}
then, for every sufficiently small $s>0$, zonoids
$Z_s$  and $Z_{-s}
$
have smooth positive generating densities and satisfy $
 Z_s+Z_{-s}=2B_2^n $ and $
 \mathcal Q_n(Z_s,Z_{-s})<1.
$
\end{lemma}

\begin{proof}
We first compute the linear coefficient.  For an endomorphism $A$ of a
finite-dimensional space,
$\det(I+sA)=1+s\operatorname{tr}A+O(s^2)$.
Consequently, the integrand in
\eqref{eq:courtade-support-volume} is
\[
 1+s\bigl(h_\eta+\operatorname{tr}A_\eta\bigr)+O(s^2).
\]
The trace is taken on the $(n-1)$-dimensional tangent space.  Therefore,
\begin{equation}\label{eq:courtade-trace}
 \operatorname{tr}A_\eta
 =\operatorname{tr}\bigl(\nabla_{\s^{n-1}}^2h_\eta\bigr)
  +(n-1)h_\eta
 =\Delta_{\s^{n-1}}h_\eta+(n-1)h_\eta.
\end{equation}
Here the first trace is the Laplace--Beltrami operator, while the factor
$n-1$ is the dimension of $T_u\s^{n-1}$.  It follows that the linear
coefficient $a_1$ of $p_\eta$ is
$$
 a_1
 =\frac{1}{n|B_2^n|}\int_{\s^{n-1}}
 \bigl(\Delta_{\s^{n-1}}h_\eta+n h_\eta\bigr)
 \,d\mathcal H^{n-1}=\frac1{|B_2^n|}\int_{\s^{n-1}}h_\eta
 \,d\mathcal H^{n-1}.
$$
In the last step, we used
$
 \int_{\s^{n-1}}\Delta_{\s^{n-1}}h_\eta
 \,d\mathcal H^{n-1}=0,
$
which follows from integration by parts on the sphere.  Fubini's theorem
and Cauchy's formula now give
$$
a_1=\frac1{|B_2^n|}\int_{\s^{n-1}}
   \int_{\s^{n-1}}|\langle x,u\rangle|\,d\eta(u)
   \,d\mathcal H^{n-1}(x)=\frac{2|B_2^{n-1}|}{|B_2^n|}\eta(\s^{n-1})=0.
$$

The generating measures of $Z_s$ and $Z_{-s}$ are $\mu_n+s\eta$ and
$\mu_n-s\eta$, respectively.  Their sum is $2\mu_n$, so
\[
 Z_s+Z_{-s}=2B_2^n.
\]
Set
\begin{equation}\label{eq:courtade-Sigma}
 \Sigma(s)=\frac{|\partial Z_s|_{n-1}}{n|B_2^n|}.
\end{equation}
The first-variation formula for volume yields
\[
 |\partial Z_s|_{n-1}
 =\left.\frac{d}{d\tau}\right|_{\tau=0^+}
 |Z_s+\tau B_2^n|;
\]
see, for example, \cite[Chapter~5]{Schneider-14}.  Equality of support
functions gives
$
 Z_s+\tau B_2^n=(1+\tau)Z_{s/(1+\tau)}$.
It follows that
\[
 |Z_s+\tau B_2^n|
 =|B_2^n|(1+\tau)^n
 p_\eta\!\left(\frac{s}{1+\tau}\right).
\]
Differentiating at $\tau=0$ gives the useful identity
\begin{equation}\label{eq:courtade-Sigma-polynomial}
 \Sigma(s)=p_\eta(s)-\frac{s}{n}p'_\eta(s).
\end{equation}
In particular,
\[
 \frac{|\partial Z_s|_{n-1}}{|Z_s|}
 =n\frac{\Sigma(s)}{p_\eta(s)},
 \qquad
 \frac{|\partial Z_{-s}|_{n-1}}{|Z_{-s}|}
 =n\frac{\Sigma(-s)}{p_\eta(-s)}.
\]
Since $|Z_s+Z_{-s}|=2^n|B_2^n|$, substitution into
\eqref{eq:courtade-Q} gives the exact formula
\begin{equation}\label{eq:courtade-Q-polynomial}
 \mathcal Q_n(Z_s,Z_{-s})
 =\left[
 \frac12\left(
 \frac{\Sigma(s)}{p_\eta(s)}+
 \frac{\Sigma(-s)}{p_\eta(-s)}
 \right)\right]^n p_\eta(s)p_\eta(-s).
\end{equation}
To extract the fourth-order term, we define
\[
 L(s)=\frac12\bigl(\log p_\eta(s)+\log p_\eta(-s)\bigr).
\]
For small $s$, the polynomial $p_\eta(s)$ is positive, so $L$ is well defined.
It is an even analytic function.  By
\eqref{eq:courtade-Sigma-polynomial},
\[
 \frac12\left(
 \frac{\Sigma(s)}{p_\eta(s)}+
 \frac{\Sigma(-s)}{p_\eta(-s)}
 \right)
 =1-\frac{s}{n}L'(s).
\]
Thus \eqref{eq:courtade-Q-polynomial} becomes
\begin{equation}\label{eq:courtade-Q-log-identity}
 \log\mathcal Q_n(Z_s,Z_{-s})
 =n\log\!\left(1-\frac{s}{n}L'(s)\right)+2L(s).
\end{equation}
The expansion \eqref{eq:courtade-P-short} gives
\[
 L(s)=a_2s^2+\left(a_4-\frac12a_2^2\right)s^4+O(s^6).
\]
Substituting this into \eqref{eq:courtade-Q-log-identity} yields
\[
 \log\mathcal Q_n(Z_s,Z_{-s})
 =\left(\frac{n-2}{n}a_2^2-2a_4\right)s^4+O(s^6).
\]
Condition \eqref{eq:courtade-fourth-short} therefore implies
$\mathcal Q_n(Z_s,Z_{-s})<1$ for every sufficiently small $s>0$.
\end{proof}

We next construct a four-dimensional perturbation satisfying
\eqref{eq:courtade-fourth-short}.  Decompose
\[
 \R^4=E_1\oplus E_2,
 \qquad
 E_1=\operatorname{span}\{e_1,e_2\},
 \qquad
 E_2=\operatorname{span}\{e_3,e_4\}.
\]
For $u=(u_1,u_2,u_3,u_4)\in\s^3$, put
\[
 q(u)=|P_{E_1}u|^2=u_1^2+u_2^2,
 \qquad
 1-q(u)=|P_{E_2}u|^2=u_3^2+u_4^2,
\]
and define
\[
 \phi(q)=q^2-2q+\frac23.
\]

We first describe the spherical coordinates that will be used twice below:
first to identify the distribution of $q$, and later to compute the
spherical Hessian of $\phi$.  Write
\begin{equation}\label{eq:courtade-bipolar-coordinates}
 u=(\cos\theta\,\xi,\sin\theta\,\zeta),
 \qquad
 0\leq\theta\leq\frac\pi2,
 \qquad
 \xi,\zeta\in\s^1.
\end{equation}
Then $q=\cos^2\theta$.  The coordinate vectors in the $\theta$,
$\xi$, and $\zeta$ directions are mutually orthogonal and have lengths
$1$, $\cos\theta$, and $\sin\theta$, respectively.  Hence
\begin{equation}\label{eq:courtade-bipolar-metric}
 ds_{\s^3}^2
 =d\theta^2+\cos^2\theta\,d\xi^2
  +\sin^2\theta\,d\zeta^2,
\end{equation}
where $d\xi^2$ and $d\zeta^2$ denote the standard metrics on the two
copies of $\s^1$.  The corresponding area element is
\begin{equation}\label{eq:courtade-bipolar-area}
 d\mathcal H^3(u)
 =\sin\theta\cos\theta\,d\theta\,
 d\mathcal H^1(\xi)\,d\mathcal H^1(\zeta).
\end{equation}
Since $\mathcal H^1(\s^1)=2\pi$ and 
 $\mathcal H^3(\s^3)=2\pi^2$,
formula \eqref{eq:courtade-bipolar-area} implies that, for every
integrable function $\Psi:[0,1]\to\R$,
\begin{equation}\label{eq:courtade-q-uniform}
\frac1{2\pi^2}\int_{\s^3}\Psi(q(u))\,d\mathcal H^3(u)=2\int_0^{\pi/2}\Psi(\cos^2\theta)  \sin\theta\cos\theta\,d\theta=\int_0^1\Psi(q)\,dq.
\end{equation}

Thus, with respect to normalized Hausdorff
measure $(2\pi^2)^{-1}\mathcal H^3$ on $\s^3$, the variable $q$ is
uniformly distributed on $[0,1]$.  More generally, the squared length of
the projection of a uniformly distributed point of $\s^{n-1}$ onto a
$k$-dimensional subspace has the beta distribution
$\operatorname{Beta}(k/2,(n-k)/2)$; the calculation above is the special
case $\operatorname{Beta}(1,1)$ and is a direct proof of the fact needed
here.

Set
\[
 Y_2(q(u))=q(u)-\frac12,
 \qquad
 Y_4(q(u))=\left(q(u)-\frac12\right)^2-\frac1{12}.
\]
Then
\[
 \phi=Y_4-Y_2.
\]
For completeness, let us verify that $Y_2$ and $Y_4$ are spherical
harmonics of degrees two and four.  For $x\in\R^4$, put
$
 D(x)=|P_{E_1}x|^2-|P_{E_2}x|^2.$
The homogeneous polynomial $D/2$ is harmonic and restricts to $Y_2$ on
$\s^3$.  Moreover,
$
 \widetilde Y_4(x)=\frac14D(x)^2-\frac1{12}|x|^4
$
restricts to $Y_4$.  Since
$
 \Delta D^2=8|x|^2$  and 
 $\Delta|x|^4=24|x|^2$ in $\R^4$, 
we have $\Delta\widetilde Y_4=0$.  Thus $\widetilde Y_4$ is homogeneous
and harmonic of degree four, as claimed.

We now find a smooth signed generating density whose cosine transform is
$\phi$.  To keep all normalizations in terms of Hausdorff measure, define
\[
 (\mathcal C\rho)(x)
 =\int_{\s^3}|\langle x,u\rangle|\rho(u)
 \,d\mathcal H^3(u).
\]
By the Funk--Hecke theorem, $\mathcal C$ acts by a scalar on each space
of spherical harmonics; see \cite[Chapter~3]{Groemer-96}.  We calculate
the two required scalars directly.  Formula
\eqref{eq:courtade-q-uniform} also shows that, conditionally on $q$,
$\xi$ and $\zeta$ in \eqref{eq:courtade-bipolar-coordinates} are
uniformly distributed on their circle factors.  Since
$u_1=\sqrt q\,\xi_1$ and
\[
 \frac1{2\pi}\int_{\s^1}|\xi_1|\,d\mathcal H^1(\xi)=\frac2\pi,
\]
we obtain, for $k=0,1,2$,
\begin{align*}
 \frac1{2\pi^2}\int_{\s^3}|u_1|q(u)^k
 \,d\mathcal H^3(u)
 &=\left(\int_0^1q^{k+1/2}\,dq\right)
   \left(\frac1{2\pi}\int_{\s^1}|\xi_1|
   \,d\mathcal H^1(\xi)\right)\\
 &=\frac4{\pi(2k+3)}.
\end{align*}
Equivalently,
\begin{equation}\label{eq:courtade-moment-table}
 \int_{\s^3}|u_1|\bigl(1,q(u),q(u)^2\bigr)
 \,d\mathcal H^3(u)
 =8\pi\left(\frac13,\frac15,\frac17\right).
\end{equation}
Evaluating the Funk--Hecke multipliers at $e_1$ and using
$Y_2(e_1)=1/2$ and $Y_4(e_1)=1/6$, we obtain
\begin{equation}\label{eq:courtade-cosine-table-short}
 \mathcal CY_2=\frac{8\pi}{15}Y_2,
 \qquad
 \mathcal CY_4=-\frac{8\pi}{105}Y_4.
\end{equation}
It follows that the smooth even function
\begin{equation}\label{eq:courtade-rho}
 \rho=-\frac{105}{8\pi}Y_4-\frac{15}{8\pi}Y_2
\end{equation}
satisfies $\mathcal C\rho=\phi$.  Nonconstant spherical harmonics have
mean zero; alternatively, this follows here by integrating $Y_2$ and
$Y_4$ over $[0,1]$ in \eqref{eq:courtade-q-uniform}.  Hence
\[
 d\eta=\rho\,d\mathcal H^3
\]
is a smooth even signed measure with $\eta(\s^3)=0$, and
\[
 h_\eta=\phi.
\]
For all sufficiently small $|s|$, the measure $\mu_4+s\eta$ is positive.
Let $Z_s$ be the zonoid it generates.  Then
\[
 h_{Z_s}=1+s\phi.
\]

It remains to compute the volume polynomial of this zonoid
family.  In the coordinates \eqref{eq:courtade-bipolar-coordinates}, an
orthonormal frame of $T_u\s^3$, away from the two endpoint circles, is
\[
 e_\theta=(-\sin\theta\,\xi,\cos\theta\,\zeta),
 \qquad
 e_\xi=(\tau_\xi,0),
 \qquad
 e_\zeta=(0,\tau_\zeta),
\]
where $\tau_\xi$ and $\tau_\zeta$ are unit tangent vectors to the two
circle factors.  Thus $e_\theta$ is the meridional direction, while
$e_\xi$ and $e_\zeta$ are the unit tangent directions to the $E_1$- and
$E_2$-circle factors, respectively.

Put $\Phi(\theta)=\phi(\cos^2\theta)$.  The warped-product metric
\eqref{eq:courtade-bipolar-metric} gives
\[
 \nabla_{\s^3}^2\Phi(e_\theta,e_\theta)=\Phi''(\theta),
 \qquad
 \nabla_{\s^3}^2\Phi(e_\xi,e_\xi)
 =-\tan\theta\,\Phi'(\theta),
\]
and
\[
 \nabla_{\s^3}^2\Phi(e_\zeta,e_\zeta)
 =\cot\theta\,\Phi'(\theta).
\]
The mixed terms vanish.  These identities follow, for example, by using
$\nabla\Phi=\Phi'e_\theta$ and differentiating in the two unit circle
directions in the metric \eqref{eq:courtade-bipolar-metric}.
Since $q=\cos^2\theta$,
\[
 q'=-2\sin\theta\cos\theta,
 \qquad
 q''=2(1-2q).
\]
The chain rule therefore yields
\begin{align*}
 \nabla_{\s^3}^2\phi(e_\theta,e_\theta)
 &=2(1-2q)\phi'(q)+4q(1-q)\phi''(q),\\
 \nabla_{\s^3}^2\phi(e_\xi,e_\xi)
 &=2(1-q)\phi'(q),\\
 \nabla_{\s^3}^2\phi(e_\zeta,e_\zeta)
 &=-2q\phi'(q).
\end{align*}
Thus the three eigenvalues of
\[
 \nabla_{\s^3}^2(1+s\phi)+(1+s\phi)I_{T_u\s^3}
\]
are $1+sr_0$, $1+sr_1$, and $1+sr_2$, where
\begin{align*}
 r_0(q)
 &=\phi+2(1-2q)\phi'+4q(1-q)\phi''
   =-15q^2+18q-\frac{10}{3},\\
 r_1(q)
 &=\phi+2(1-q)\phi'
   =-3q^2+6q-\frac{10}{3},\\
 r_2(q)
 &=\phi-2q\phi'
   =-3q^2+2q+\frac23.
\end{align*}
Although the above frame degenerates at $q=0$ and $q=1$, the functions
$r_0,r_1,r_2$ extend continuously to both endpoints.

Since $4|B_2^4|=2\pi^2=\mathcal H^3(\s^3)$, formulas
\eqref{eq:courtade-support-volume} and \eqref{eq:courtade-q-uniform}
give
\begin{equation}\label{eq:courtade-4d-integral}
 \frac{|Z_s|}{|B_2^4|}
 =\int_0^1(1+s\phi)(1+sr_0)(1+sr_1)(1+sr_2)\,dq.
\end{equation}
Only the coefficients of $s^2$ and $s^4$ will be needed.  The coefficient
of $s^2$ is
\begin{align*}
 a_2^{(4)}
 &=\int_0^1\bigl[
 \phi(r_0+r_1+r_2)+r_0r_1+r_0r_2+r_1r_2
 \bigr]\,dq\\
 &=\int_0^1\left(
 \frac83-48q+152q^2-184q^3+78q^4
 \right)dq
 =-\frac{16}{15},
\end{align*}
whereas the coefficient of $s^4$ is
\begin{align*}
 a_4^{(4)}
 &=\int_0^1\phi r_0r_1r_2\,dq\\
 &=\int_0^1
 \left(q^2-2q+\frac23\right)
 \left(-15q^2+18q-\frac{10}{3}\right)
 \left(-3q^2+6q-\frac{10}{3}\right)
 \left(-3q^2+2q+\frac23\right)dq\\
 &=\frac{436}{945}.
\end{align*}
 For completeness, we note that  $a_3^{(4)}=208/945$ and  the linear coefficient is zero.  Thus
\begin{align}
 \frac{|Z_s|}{|B_2^4|}
 &=1-\frac{16}{15}s^2+
   \frac{208}{945}s^3+\frac{436}{945}s^4.
 \label{eq:courtade-4d-short}
\end{align}

We now transfer the two relevant coefficients to every dimension
$n>4$.  We do this at the level of scalar determinant integrals and only
afterward smooth the resulting signed measure.  Recall that if a zonoid
$Z\subset\R^n$ has generating measure $\mu$, then
\begin{equation}\label{eq:courtade-zonoid-determinant}
 |Z|=\frac{2^n}{n!}
 \int_{(\s^{n-1})^n}|\det(u_1,\ldots,u_n)|
 \,d\mu(u_1)\cdots d\mu(u_n);
\end{equation}
see \cite[Theorem~5.3.2]{Schneider-14}.  Let $\nu$ be a finite even
signed measure on $\s^{n-1}$.  Expanding
\eqref{eq:courtade-zonoid-determinant} for the measure $\mu_n+s\nu$
shows that its coefficient of $s^j$, divided by $|B_2^n|$, is
\begin{align}
 \alpha_j^{(n)}(\nu)
 &=\frac{2^n}{j!(n-j)!|B_2^n|}
 \int_{(\s^{n-1})^n}|\det(u_1,\ldots,u_n)|
 \prod_{i=1}^j d\nu(u_i)
 \prod_{i=j+1}^n d\mu_n(u_i).
 \label{eq:courtade-alpha-nfold}
\end{align}
For a signed $\nu$, this is simply a scalar integral.  If
$\mu_n+s\nu$ is positive, however, it is exactly a coefficient of the
volume polynomial of the zonoid generated by that measure.

We next integrate out the last $n-j$ variables in
\eqref{eq:courtade-alpha-nfold}.  Fix linearly independent
$u_1,\ldots,u_j$ and let $L=\operatorname{span}\{u_1,\ldots,u_j\}$.
Orthogonal projection onto $L^\perp$ gives
\begin{align*}
 |\det(u_1,\ldots,u_j,v_{j+1},\ldots,v_n)|=
 |u_1\wedge\cdots\wedge u_j|
 \left|\det_{L^\perp}\bigl(
 P_{L^\perp}v_{j+1},\ldots,P_{L^\perp}v_n
 \bigr)\right|.
\end{align*}
The pushforward of $\mu_n$ under $P_{L^\perp}$ generates the Euclidean
ball $B_2^{n-j}\subset L^\perp$, because for $y\in L^\perp$,
\[
 \int|\langle y,P_{L^\perp}v\rangle|\,d\mu_n(v)
 =\int|\langle y,v\rangle|\,d\mu_n(v)=|y|.
\]
Denote this pushforward measure by \(\lambda\). Since orthogonal
projection may decrease lengths, \(\lambda\) is supported on the unit
ball of \(L^\perp\), rather than necessarily on its unit sphere. To put
the measure into the spherical form required in \eqref{eq:courtade-zonoid-determinant},
define a measure \(\widetilde\lambda\) on
\(\s^{n-1}\cap L^\perp\) by
\[
 \int_{\s^{n-1}\cap L^\perp}\psi(\theta)\,
 d\widetilde\lambda(\theta)
 =
 \int_{L^\perp\setminus\{0\}}
 |w|\,\psi\!\left(\frac{w}{|w|}\right)\,d\lambda(w).
\]
The mass of \(\lambda\) at the origin may be discarded, since it
contributes neither to the support function nor to the determinant
integral. By homogeneity,
\[
 \int_{L^\perp}|\langle y,w\rangle|\,d\lambda(w)
 =
 \int_{\s^{n-1}\cap L^\perp}
 |\langle y,\theta\rangle|\,d\widetilde\lambda(\theta),
\]
so \(\widetilde\lambda\) is a spherical generating measure of
\(B_2^{n-j}\subset L^\perp\). Moreover, if \(d=n-j\), then
\[
\begin{aligned}
 &\int_{(L^\perp)^d}
 \bigl|\det_{L^\perp}(w_1,\ldots,w_d)\bigr|\,
 d\lambda(w_1)\cdots d\lambda(w_d)\\
 &\qquad =
 \int_{(\s^{n-1}\cap L^\perp)^d}
 \bigl|\det_{L^\perp}(\theta_1,\ldots,\theta_d)\bigr|\,
 d\widetilde\lambda(\theta_1)\cdots
 d\widetilde\lambda(\theta_d),
\end{aligned}
\]
because each determinant is homogeneous of degree one in each of its
arguments. Applying \eqref{eq:courtade-zonoid-determinant} in the
\(d=n-j\) dimensional space \(L^\perp\), we conclude that the integral
over the last \(n-j\) variables in \eqref{eq:courtade-alpha-nfold}
equals $
 \frac{(n-j)!}{2^{n-j}}\,|B_2^{n-j}|$.

If $u_1,\ldots,u_j$ are linearly dependent, both sides of the resulting
formula vanish.  Consequently,
\begin{equation}\label{eq:courtade-alpha-wedge}
 \alpha_j^{(n)}(\nu)
 =\frac{2^j|B_2^{n-j}|}{j!|B_2^n|}
 \int_{(\s^{n-1})^j}|u_1\wedge\cdots\wedge u_j|
 \,d\nu(u_1)\cdots d\nu(u_j).
\end{equation}

Regarding the four-dimensional measure
$d\eta=\rho\,d\mathcal H^3$ constructed above as a signed measure on
$\s^{n-1}$ supported on $E\cap\s^{n-1}$, where $E$ is a fixed
four-dimensional subspace of $\R^n$.  The norm of a wedge product is
unchanged by this isometric inclusion.  Comparing
\eqref{eq:courtade-alpha-wedge} in dimensions $4$ and $n$, we obtain,
for $0\leq j\leq4$,
\begin{equation}\label{eq:courtade-dimension-ratio}
 \alpha_j^{(n)}(\eta)
 =\frac{|B_2^4||B_2^{n-j}|}
        {|B_2^{4-j}||B_2^n|}\,
   \alpha_j^{(4)}(\eta).
\end{equation}
At this stage, the embedded signed measure $\eta$ is used only to compute
the limiting scalar coefficients in
\eqref{eq:courtade-alpha-wedge}; we do not associate a convex body with
this singular signed measure.

From \eqref{eq:courtade-4d-short},
\[
 \alpha_2^{(4)}(\eta)=-\frac{16}{15},
 \qquad
 \alpha_4^{(4)}(\eta)=\frac{436}{945}.
\]
Then \eqref{eq:courtade-dimension-ratio} gives
\begin{equation}\label{eq:courtade-transferred-coefficients}
 \alpha_2^{(n)}(\eta)=-\frac{4n}{15},
 \qquad
 \alpha_4^{(n)}(\eta)=\frac{109n(n-2)}{1890}.
\end{equation}
Therefore,
\begin{equation}\label{eq:courtade-negative-short}
 \frac{n-2}{n}\bigl(\alpha_2^{(n)}(\eta)\bigr)^2
 -2\alpha_4^{(n)}(\eta)
 =-\frac{209}{4725}n(n-2)<0.
\end{equation}

For $n=4$, the measure $d\eta=\rho\,d\mathcal H^3$ already has a
smooth density, so Lemma~\ref{lem:courtade-perturbation} applies
directly.  Suppose now that $n>4$.  The embedded measure $\eta$ is
supported on the lower-dimensional sphere $E\cap\s^{n-1}$ and is
therefore singular with respect to $\mathcal H^{n-1}$.  We approximate it
by smooth measures before applying the lemma.

For $0<r<1$, let
\[
 P_r(u,v)=\frac{1-r^2}{\mathcal H^{n-1}(\s^{n-1})|u-rv|^n},
 \qquad u,v\in\s^{n-1},
\]
be the spherical Poisson kernel.  Define
\begin{equation}\label{eq:courtade-poisson-smoothing}
 f_r(u)=\int_{\s^{n-1}}P_r(u,v)\,d\eta(v)
 \mbox{  and  }
 d\eta_r(u)=f_r(u)\,d\mathcal H^{n-1}(u).
\end{equation}
For each $r<1$, the density $f_r$ is smooth.  The kernel is nonnegative,
rotation invariant, and satisfies
\[
 \int_{\s^{n-1}}P_r(u,v)\,d\mathcal H^{n-1}(u)=1.
\]
It follows that $\eta_r$ is even and
\[
 \eta_r(\s^{n-1})=\eta(\s^{n-1})=0,
 \qquad
 \|\eta_r\|_{\mathrm{TV}}\leq\|\eta\|_{\mathrm{TV}}.
\]
The Poisson kernels form an approximate identity, and hence
\begin{equation}\label{eq:courtade-poisson-weak}
 \eta_r\longrightarrow\eta
 \quad\text{weakly as }r\uparrow1;
\end{equation}
see, for example, \cite[Chapter~3]{Groemer-96} for spherical Poisson
approximation.

For fixed $j$, the function
\[
 (u_1,\ldots,u_j)\longmapsto
 |u_1\wedge\cdots\wedge u_j|
\]
is continuous and bounded on the compact space $(\s^{n-1})^j$.
Moreover, weak convergence in \eqref{eq:courtade-poisson-weak} and the
uniform total-variation bound imply
$\eta_r^{\otimes j}\to\eta^{\otimes j}$ weakly.  Indeed, this is immediate
for products of continuous functions; finite sums of such products are
dense in $C((\s^{n-1})^j)$ by the Stone--Weierstrass theorem, and the
uniform total-variation bound controls the approximation error.  Formula
\eqref{eq:courtade-alpha-wedge} therefore gives
\[
 \alpha_j^{(n)}(\eta_r)\longrightarrow
 \alpha_j^{(n)}(\eta),
 \qquad j=2,4.
\]
Because the inequality in \eqref{eq:courtade-negative-short} is strict,
we may fix $r<1$ sufficiently close to $1$ so that
\[
 \frac{n-2}{n}\bigl(\alpha_2^{(n)}(\eta_r)\bigr)^2
 -2\alpha_4^{(n)}(\eta_r)<0.
\]

For this fixed $r$, the measures
\[
 \mu_n+s\eta_r=(m_n+sf_r)\mathcal H^{n-1}
\]
are positive for all sufficiently small $|s|$ and therefore generate
 zonoids.  By
\eqref{eq:courtade-zonoid-determinant}--\eqref{eq:courtade-alpha-wedge},
the coefficients $a_2$ and $a_4$ of their volume polynomial are precisely
$\alpha_2^{(n)}(\eta_r)$ and $\alpha_4^{(n)}(\eta_r)$.  Applying
Lemma~\ref{lem:courtade-perturbation}, we obtain, for every sufficiently
small $\varepsilon>0$,  zonoids
$B_\varepsilon=Z_\varepsilon, C_\varepsilon=Z_{-\varepsilon}$, with smooth positive generating densities
such that
\[
 B_\varepsilon+C_\varepsilon=2B_2^n
 \qquad\text{and}\qquad
 \mathcal Q_n(B_\varepsilon,C_\varepsilon)<1.
\]
By Proposition~\ref{prop:courtade-surface}, every sufficiently large
common dilation of $B_\varepsilon$ and $C_\varepsilon$ violates
Courtade's inequality.  This proves
Theorem~\ref{thm:courtade-counterexample}(ii).  Together with the
three-dimensional construction in the preceding subsection, this
completes the proof of the theorem.

\section*{Acknowledgments and AI assistance disclosure}

The authors thank LAMA at Université Gustave Eiffel and the Institut Mathématique de Jussieu for their hospitality. 

Parts of this work were completed while A. M., S. W., and A. Z. were participating in the program ``Synergies between Geometry, Probability, and Computation in High Dimensions'' at the Institute for Computational and Experimental Research in Mathematics (ICERM), which is supported by the National Science Foundation under Grant DMS-2424556.

A. M. and A. Z. are supported in part by the U.S. National Science Foundation Grant DMS-2247771 and the United States--Israel Binational Science Foundation (BSF) Grant 2018115.

A. M. was supported by the Chateaubriand Fellowship of the Office for Science \& Technology of the Embassy of France in the United States.

During the preparation of this paper,   OpenAI's GPT-5.6 Sol  was used as an auxiliary tool to explore examples, test determinant computations, and assist with preliminary proof development. The mathematical arguments, final statements, and computations presented here are those of the authors, who have verified them and take full responsibility for the content of the paper.

\bibliographystyle{siam}
% \bibliography{references}

\begin{thebibliography}{10}

\bibitem{AS-25}
{\sc G.~Averkov and I.~Soprunov}, {\em An algebraic-combinatorial proof of a {B}{\'e}zout-type inequality for mixed volumes of three-dimensional zonoids}, Discrete \& Computational Geometry,  (2025).

\bibitem{ADRS-24}
{\sc G.~Averkov, K.~von Dichter, S.~Richard, and I.~Soprunov}, {\em Mixed volumes of zonoids and the absolute value of the {G}rassmannian}.
\newblock arXiv:2404.02842.

\bibitem{ADS}
{\sc G.~Averkov, K.~von Dichter, and I.~Soprunov}, {\em On the log-submodularity for zonoids: from mixed volume inequalities to the hypercube}.
\newblock In preparation.

\bibitem{BM-12}
{\sc S.~Bobkov and M.~Madiman}, {\em Reverse {B}runn-{M}inkowski and reverse entropy power inequalities for convex measures}, J. Funct. Anal., 262 (2012), pp.~3309--3339.

\bibitem{B-1929}
{\sc T.~Bonnesen}, {\em Les probl{\`e}mes des isop{\'e}rim{\`e}tres et des is{\'e}piphanes}, Gauthier-Villars, Paris, 1929.

\bibitem{CW}
{\sc Y.-B. Choe and D.~G. Wagner}, {\em Rayleigh matroids}, Combinatorics, Probability and Computing, 15 (2006), pp.~765--781.

\bibitem{DCT-91}
{\sc A.~Dembo, T.~M. Cover, and J.~A. Thomas}, {\em Information-theoretic inequalities}, IEEE Trans. Inform. Theory, 37 (1991), pp.~1501--1518.

\bibitem{Fen36}
{\sc W.~Fenchel}, {\em {Generalisation du theoreme de Brunn et Minkowski concernant les corps convexes}}, C. R. Acad. Sci. Paris., 203 (1936), pp.~764--766.

\bibitem{FHMNWZ-26-II}
{\sc M.~Fradelizi, A.~Hubard, A.~Manui, C.~S. Ndiaye, S.~Wang, and A.~Zvavitch}, {\em Volume and projection inequalities {II}: {$L_p$}-sums}.
\newblock Manuscript, 2026.

\bibitem{FMMZ-24}
{\sc M.~Fradelizi, M.~Madiman, M.~Meyer, and A.~Zvavitch}, {\em On the volume of the {M}inkowski sum of zonoids}, J. Funct. Anal., 286 (2024), pp.~Paper No. 110247, 41.

\bibitem{FMZ-24}
{\sc M.~Fradelizi, M.~Madiman, and A.~Zvavitch}, {\em Sumset estimates in convex geometry}, Int. Math. Res. Not. IMRN,  (2024), pp.~11426--11454.

\bibitem{Groemer-96}
{\sc H.~Groemer}, {\em Geometric Applications of Fourier Series and Spherical Harmonics}, vol.~61 of Encyclopedia of Mathematics and its Applications, Cambridge University Press, Cambridge, 1996.

\bibitem{MMZZ-26}
{\sc M.~Madiman, A.~Manui, B.~Zawalski, and A.~Zvavitch}, {\em Integral inequalities for $\alpha$-convolutions of $\alpha$-concave functions}, 2026.
\newblock arXiv:2608.10456.

\bibitem{MNZ-25}
{\sc A.~Manui, C.~S. Ndiaye, and A.~Zvavitch}, {\em On the volume of sums of anti-blocking convex bodies}, Communications in Contemporary Mathematics, 22 (2025), p.~2550095.

\bibitem{NT-17}
{\sc P.~Nayar and T.~Tkocz}.
\newblock Personal communication, 2017.

\bibitem{Oxley}
{\sc J.~G. Oxley}, {\em Matroid Theory}, vol.~21 of Oxford Graduate Texts in Mathematics, Oxford University Press, Oxford, 2~ed., 2011.

\bibitem{H-70}
{\sc H.~Pl\"unnecke}, {\em Eine zahlentheoretische {A}nwendung der {G}raphentheorie}, J. Reine Angew. Math., 243 (1970), pp.~171--183.

\bibitem{I-89}
{\sc I.~Z. Ruzsa}, {\em An application of graph theory to additive number theory}, Sci. Ser. A Math. Sci. (N.S.), 3 (1989), pp.~97--109.

\bibitem{Ruz97}
\leavevmode\vrule height 2pt depth -1.6pt width 23pt, {\em The {B}runn-{M}inkowski inequality and nonconvex sets}, Geom. Dedicata, 67 (1997), pp.~337--348.

\bibitem{Schneider-14}
{\sc R.~Schneider}, {\em Convex Bodies: The {B}runn--{M}inkowski Theory}, vol.~151 of Encyclopedia of Mathematics and its Applications, Cambridge University Press, Cambridge, second expanded~ed., 2014.

\bibitem{Schneider-Weil-1983}
{\sc R.~Schneider and W.~Weil}, {\em Zonoids and related topics}, in Convexity and Its Applications, P.~M. Gruber and J.~M. Wills, eds., Birkh{\"a}user, Basel, 1983, pp.~296--317.

\bibitem{S-26}
{\sc R.~Skorupinski}, {\em Zonoid volumes are not log-submodular}, 2026.
\newblock arXiv:2608.07702.

\bibitem{TV06:book}
{\sc T.~Tao and V.~H. Vu}, {\em Additive combinatorics}, vol.~105 of Cambridge Studies in Advanced Mathematics, Cambridge University Press, Cambridge, 2006.

\bibitem{wagner}
{\sc D.~G. Wagner}, {\em Rank-three matroids are {Rayleigh}}, The Electronic Journal of Combinatorics, 12 (2005), p.~N8.

\bibitem{Witsenhausen-1978}
{\sc H.~S. Witsenhausen}, {\em A support characterization of zonotopes}, Mathematika, 25 (1978), pp.~13--16.

\bibitem{Z-95}
{\sc G.~M. Ziegler}, {\em Lectures on polytopes}, vol.~152 of Graduate Texts in Mathematics, Springer-Verlag, New York, 1995.

\end{thebibliography}

\bigskip
\noindent Matthieu Fradelizi
\\
Univ Gustave Eiffel, Univ Paris Est Creteil, CNRS, LAMA UMR8050, F-77447 Marne-la-Vall\'ee, France.
\\
E-mail address: matthieu.fradelizi@univ-eiffel.fr
\vspace{2mm}
\\
\noindent Alfredo Hubard
\\
Laboratoire d'Informatique Gaspard Monge, Univ Gustave Eiffel, Univ Paris Est Creteil, 77447 Marne-la-Vall\'ee, France.
\\
E-mail address: alfredo.hubard@univ-eiffel.fr
\vspace{2mm}
\\
\noindent Auttawich Manui 
\\
Department of Mathematical Sciences, Kent State University, Kent, OH 44242, USA.
\\
E-mail address: amanui@kent.edu
\vspace{2mm}
\\
\noindent Cheikh Saliou Ndiaye
\\
Univ Gustave Eiffel, Univ Paris Est Creteil, CNRS, LAMA UMR8050, F-77447 Marne-la-Vall\'ee, France.
\\
E-mail address: cheikh-saliou.ndiaye@univ-eiffel.fr
\vspace{2mm}
\\
\noindent Shouda Wang
\\
Applied and Computational Mathematics, Princeton University, Princeton, NJ 08544, USA.
\\
E-mail address:  shoudawang@princeton.edu
\vspace{2mm}
\\
\noindent Artem Zvavitch
\\
Department of Mathematical Sciences, Kent State University, Kent, OH 44242, USA. 
\\ E-mail address: zvavitch@math.kent.edu

\end{document}